\documentclass[final,5p,times,8pt]{elsarticle}
\usepackage[centertags]{amsmath}
\usepackage{amsmath,amsfonts,amssymb,amsthm,epstopdf,newlfont,graphicx}
\usepackage{fixltx2e,booktabs,bm,enumitem,setspace,empheq,cancel,tikz}
\usetikzlibrary{shapes.geometric, arrows}
\usepackage{tabularx}
\usepackage{longtable}
\usepackage[raggedrightboxes]{ragged2e}
\usepackage{algorithm}
\usepackage{float}
\usepackage[short,nocomma]{optidef}
\usepackage{algpseudocode}
\usepackage{subcaption}
\usepackage[labelfont=bf,format=plain,justification=raggedright,singlelinecheck=false]{caption}
\usepackage{threeparttable,multirow}
\usepackage{amsmath,accents}
\usepackage[LGR,T1]{fontenc}
\newcommand{\test}[1]{{\fontencoding{LGR}\fontfamily{#1}\selectfont κ}}
\newcommand{\Kappa}{\test{porson}}
\usepackage[american]{babel}
\usepackage{natbib}
\biboptions{comma,square,sort&compress}
\usepackage{appendix}
\usepackage[colorlinks=true]{hyperref}

\newcommand\hcancel[2][black]{\setbox0=\hbox{$#2$}%
\rlap{\raisebox{.25\ht0}{\textcolor{#1}{\rule{0.7\wd0}{0.75pt}}}}#2} 
\newcommand\hcancelt[2][black]{\setbox0=\hbox{$#2$}%
\rlap{\raisebox{.25\ht0}{\textcolor{#1}{\hspace{0.3mm}\rule{0.7\wd0}{0.75pt}}}}#2} 
\newtheorem{thm}{Theorem}[section]

\newtheorem{cor}{Corollary}[section]

\theoremstyle{definition}

\numberwithin{algorithm}{section}
\numberwithin{equation}{section}
\renewcommand{\theequation}{\thesection.\arabic{equation}}
\def\simgt{\,\hbox{\lower0.6ex\hbox{$>$}\llap{\raise0.3ex\hbox{$\sim$}}}\,}
\def\simlt{\,\hbox{\lower0.6ex\hbox{$<$}\llap{\raise0.3ex\hbox{$\sim$}}}\,}
\def\simgteq{\,\hbox{\lower0.6ex\hbox{$\ge$}\llap{\raise0.6ex\hbox{$\sim$}}}\,}
\def\simlteq{\,\hbox{\lower0.6ex\hbox{$\le$}\llap{\raise0.6ex\hbox{$\sim$}}}\,}
\def\applteq{\,\hbox{\lower0.6ex\hbox{$\le$}\llap{\raise0.8ex\hbox{$\approx$}}}\,}
\def\applt{\,\hbox{\lower0.6ex\hbox{$<$}\llap{\raise0.5ex\hbox{$\approx$}}}\,}
\DeclareMathAlphabet\mathbfcal{OMS}{cmsy}{b}{n}
\DeclareMathAlphabet{\mathpzc}{OT1}{pzc}{m}{it}
\DeclareMathAlphabet\euscr{U}{eus}{m}{n}

\makeatletter
\def\user@resume{resume}
\def\user@intermezzo{intermezzo}
\newcounter{previousequation}
\newcounter{lastsubequation}
\newcounter{savedparentequation}
\renewenvironment{subequations}[1][]{%
      \def\user@decides{#1}%
      \setcounter{previousequation}{\value{equation}}%
      \ifx\user@decides\user@resume 
           \setcounter{equation}{\value{savedparentequation}}%
      \else  
      \ifx\user@decides\user@intermezzo
           \refstepcounter{equation}%
      \else
           \setcounter{lastsubequation}{0}%
           \refstepcounter{equation}%
      \fi\fi
      \protected@edef\theHparentequation{%
          \@ifundefined {theHequation}\theequation \theHequation}%
      \protected@edef\theparentequation{\theequation}%
      \setcounter{parentequation}{\value{equation}}%
      \ifx\user@decides\user@resume 
           \setcounter{equation}{\value{lastsubequation}}%
         \else
           \setcounter{equation}{0}%
      \fi
      \def\theequation  {\theparentequation  \alph{equation}}%
      \def\theHequation {\theHparentequation \alph{equation}}%
      \ignorespaces
}{%
  \ifx\user@decides\user@resume
       \setcounter{lastsubequation}{\value{equation}}%
       \setcounter{equation}{\value{previousequation}}%
  \else
  \ifx\user@decides\user@intermezzo
       \setcounter{equation}{\value{parentequation}}%
  \else
       \setcounter{lastsubequation}{\value{equation}}%
       \setcounter{savedparentequation}{\value{parentequation}}%
       \setcounter{equation}{\value{parentequation}}%
  \fi\fi
  \ignorespacesafterend
}
\makeatother
\newcommand{\C}[1]{\mathcal{#1}}

\newcommand{\F}[1]{\mathbf{#1}}

\newcommand{\MB}[1]{\mathbb{#1}}

\newcommand{\ME}[1]{\euscr{#1}}
\newcommand{\MBS}{\MB{S}}
\newcommand{\MBG}{\MB{G}}
\newcommand{\MBSG}{\hat{\MBG}}
\newcommand{\MBR}{\mathbb{R}}

\newcommand{\MBRP}{\MBR^+}
\newcommand{\MBRzer}{\MBR_0}
\newcommand{\MBRzerP}{\MBRzer^+}
\newcommand{\MBZ}{\mathbb{Z}}
\newcommand{\MBZP}{\MBZ^+}
\newcommand{\MBZzer}{\MBZ_0}
\newcommand{\MBZzerP}{\MBZzer^+}
\newcommand{\MBZe}{\MBZ_e}
\newcommand{\MBZeP}{\MBZe^+}
\newcommand{\MBZzereP}{\MBZ_{0,e}^+}
\newcommand{\MBZOP}{\MBZ_{o}^+}
\newcommand{\MBT}{\mathbb{T}}
\newcommand{\MBJ}{\mathbb{J}}
\newcommand{\MBJP}{\mathbb{J}^+}
\newcommand{\MBK}{\mathbb{K}}
\newcommand{\MBKcanc}[1]{\MBK_{#1,\hcancel{0}}}
\newcommand{\MBN}{\mathbb{N}}
\newcommand{\MFF}{\mathfrak{F}}
\newcommand{\MFP}{\mathfrak{P}}
\newcommand{\MFC}{\mathfrak{C}}

\newcommand{\cancbra}[1]{\hcancel{[}#1\hcancelt{]}}
\newcommand{\sumd}{\sideset{}{'}}

\newcommand{\hu}{\hat{u}}

\newcommand{\hatt}{\hat{t}}
\newcommand{\bmx}{\bm{x}}

\newcommand{\bmt}{\bm{t}}

\newcommand{\bmv}{\bm{v}}

\newcommand{\tilu}{\tilde{u}}

\newcommand{\tilpsi}{\tilde{\psi}}
\newcommand{\homega}{\hat{\omega}}

\newcommand{\bmy}{\bm{y}}

\newcommand{\bmh}{\bm{h}}

\newcommand{\bmzer}{\bm{\mathit{0}}}
\newcommand{\bmone}{\bm{\mathit{1}}}

\newcommand{\hz}{\hat{z}}

\newcommand{\hsigma}{\hat{\sigma}}
\newcommand{\hlambdabar}{\hat \lambdabar}

\newcommand{\FOmega}{\F{\Omega}}
\newcommand{\IFOmega}{\FOmega^{\circ}}

\newcommand{\foralla}{\,\forall_{\mkern-6mu a}\,}
\newcommand{\forallaa}{\,\forall_{\mkern-6mu aa}\,}
\newcommand{\foralle}{\,\forall_{\mkern-6mu e}\,}
\newcommand{\foralls}{\,\forall_{\mkern-6mu s}\,}
\newcommand{\foralll}{\,\forall_{\mkern-4mu l}\,}
\newcommand{\Def}[1]{\text{Def}\left(#1\right)}

\newcommand{\resh}[3]{\text{resh}_{#1,#2}\left(#3\right)}
\newcommand{\reshs}[2]{\text{resh}_{#1}\left(#2\right)}

\newcommand{\multrow}[1]{\begin{tabular}{@{}c@{}} #1 \end{tabular}}
\newcommand{\sigmabar}{\mathord{\sigma\kern-0.6em\raisebox{-1ex}{$\bar{\phantom{\sigma}}$}}}
\newcommand{\sigmabarmax}{\mathord{\sigma_{\max}\kern-1.9em\raisebox{-0.8ex}{$\bar{\phantom{\sigma}}$}}\;\;\;\;\;}
\newcommand{\sigmabarmin}{\mathord{\sigma_{\min}\kern-1.75em\raisebox{-0.8ex}{$\bar{\phantom{\sigma}}$}}\;\;\;\;\;}
\usepackage{mathtools}
\mathtoolsset{showonlyrefs}
\makeatletter
\def\BState{\State\hskip-\ALG@thistlm}
\makeatother
\MakeRobust{\eqref}
\errorcontextlines\maxdimen

\makeatletter
    \newcommand*{\algrule}[1][\algorithmicindent]{\makebox[#1][l]{\hspace*{.5em}\thealgruleextra\vrule height \thealgruleheight depth \thealgruledepth}}%
\newcommand*{\thealgruleextra}{}
\newcommand*{\thealgruleheight}{.75\baselineskip}
\newcommand*{\thealgruledepth}{.25\baselineskip}

\newcount\ALG@printindent@tempcnta
\def\ALG@printindent{%
    \ifnum \theALG@nested>0
        \ifx\ALG@text\ALG@x@notext
        \else
            \unskip
            \addvspace{-1pt}
            \ALG@printindent@tempcnta=1
            \loop
                \algrule[\csname ALG@ind@\the\ALG@printindent@tempcnta\endcsname]%
                \advance \ALG@printindent@tempcnta 1
            \ifnum \ALG@printindent@tempcnta<\numexpr\theALG@nested+1\relax
            \repeat
        \fi
    \fi
    }%
\usepackage{etoolbox}
\patchcmd{\ALG@doentity}{\noindent\hskip\ALG@tlm}{\ALG@printindent}{}{\errmessage{failed to patch}}
\makeatother

\newbox\statebox
\newcommand{\myState}[1]{%
    \setbox\statebox=\vbox{#1}%
    \edef\thealgruleheight{\dimexpr \the\ht\statebox+1pt\relax}%
    \edef\thealgruledepth{\dimexpr \the\dp\statebox+1pt\relax}%
    \ifdim\thealgruleheight<.75\baselineskip
        \def\thealgruleheight{\dimexpr .75\baselineskip+1pt\relax}%
    \fi
    \ifdim\thealgruledepth<.25\baselineskip
        \def\thealgruledepth{\dimexpr .25\baselineskip+1pt\relax}%
    \fi
    \State #1%
    \def\thealgruleheight{\dimexpr .75\baselineskip+1pt\relax}%
    \def\thealgruledepth{\dimexpr .25\baselineskip+1pt\relax}%
}
\makeatletter
\newcommand{\oset}[3][0ex]{%
  \mathrel{\mathop{#3}\limits^{
    \vbox to#1{\kern-2\ex@
    \hbox{$\scriptstyle#2$}\vss}}}}
\makeatother

\makeatletter
\def\ps@pprintTitle{%
  \let\@oddhead\@empty
  \let\@evenhead\@empty
  \let\@oddfoot\@empty
  \let\@evenfoot\@oddfoot
}
\makeatother
\begin{document}
\begin{frontmatter}
\title{Fourier-Gegenbauer Integral-Galerkin Method for Solving the Advection‐Diffusion Equation With Periodic Boundary Conditions}
\author[Ajman]{Kareem T. Elgindy\corref{cor1}}
\ead{k.elgindy@ajman.ac.ae}
\address[Ajman]{Department of Mathematics and Sciences, College of Humanities and Sciences, Ajman University, P.O.Box: 346 Ajman, United Arab Emirates}
\cortext[cor1]{Corresponding author}
\begin{abstract}
This study presents the Fourier-Gegenbauer Integral-Galerkin (FGIG) method, a novel and efficient numerical framework for solving the one-dimensional advection-diffusion equation with periodic boundary conditions. The FGIG method uniquely combines Fourier series for spatial periodicity and Gegenbauer polynomials for temporal integration within a Galerkin framework, resulting in highly accurate numerical and semi-analytical solutions. Distinctively, this approach eliminates the need for time-stepping procedures by reformulating the problem as a system of integral equations, reducing error accumulation over long-time simulations and improving computational efficiency. Key contributions include exponential convergence rates for smooth solutions, robustness under oscillatory conditions, and an inherently parallelizable structure, enabling scalable computation for large-scale problems. Additionally, the method introduces a barycentric formulation of shifted-Gegenbauer-Gauss quadrature to ensure high accuracy and stability for relatively low P\'{e}clet numbers. Numerical experiments validate the method's superior performance over traditional techniques, demonstrating its potential for extending to higher-dimensional problems and diverse applications in computational mathematics and engineering.
\end{abstract}

\end{frontmatter}
\section{Introduction}
\label{Int}
The advection-diffusion (AD) equation is a fundamental PDE governing the transport of a scalar quantity within a fluid medium. This scalar quantity may represent various physical properties such as heat, pollutants, or chemical concentrations. The AD equation encapsulates the combined influence of advection, the transport driven by fluid flow, and diffusion, the spreading arising from random molecular motion. This equation finds wide-ranging applications in fields such as fluid mechanics, environmental science, biophysics, and materials science. Accurate solutions to the AD equation are indispensable for critical applications such as environmental monitoring and control, engineering design, climate modeling, and medical applications. Consequently, the AD equation is a crucial tool for understanding and predicting the behavior of various natural and engineered systems.

Due to its significance, over the years, researchers have developed a wide array of numerical methods to address the AD equation, each tailored to specific applications, stability requirements, or computational constraints. In the following, we categorize and summarize some notable advancements and methods in this field through the past decade. Each method has its strengths and weaknesses, and the choice of the most appropriate method depends on the specific problem and desired accuracy. The methods are grouped based on their underlying principles and ordered chronologically to provide a clear progression of developments in this domain: \textbf{Finite Difference (FD) Methods (FDMs).} \citet{appadu2017computational} compared the upwind, non-standard FD, and unconditionally positive FD schemes for linear and nonlinear AD-reaction problems, and analyzed their dissipative and dispersive properties. \citet{khalsaraei2017efficient} evaluated and improved positivity in FDMs, which were applied to solve parabolic equations that model advection-diffusion processes coupled with reaction terms. \citet{solis2017numerical} presented non-standard FD schemes for modeling infection dynamics. \citet{al2024numerical} developed a 1D model for instantaneous spills in river systems by numerically solving the advection-dispersion equation along with the shallow water equations using FDMs. \textbf{Finite Element Methods (FEMs).} \citet{cerfontaine2016formulation} designed FEM for borehole heat exchangers. \citet{wang2022discrete} introduced nonlinear correction techniques ensuring discrete extremum principles. \citet{wang2024discrete} presented a nonlinear correction technique for FEMs applied to AD problems. \citet{jena2023one} utilized an improvised quartic-order cubic B-spline approach combined with the Crank-Nicolson method and FEM. \textbf{Finite Volume Methods (FVMs).} \citet{sejekan2016improving} enhanced diffusive flux accuracy. \citet{chernyshenko2017hybrid} combined FVM-FEM for fractured media. \citet{kramarenko2017nonlinear} developed a nonlinear correction for subsurface flows. \citet{hwang2023efficient} utilized a second-order, positivity-preserving central-upwind scheme based on the FVM to create a numerical scheme that captures contact discontinuities in scalar transport within a shallow water flow environment. \citet{mei2024unified} proposed a unified finite volume physics informed neural network  approach to improve the accuracy and efficiency of solving heterogeneous PDEs, which combines sub-domain decomposition, finite volume discretization, and conventional numerical solvers. \textbf{Semi-Lagrangian Methods.} \citet{bonaventura2016flux} proposed a fully conservative, flux-form discretization of linear and nonlinear diffusion equations using a semi-Lagrangian approach. \citet{bokanowski2016semi} developed high-order, unconditionally stable schemes for first- and second-order PDEs, including AD, using a semi-Lagrangian approach with discontinuous Galerkin (DG) elements. \citet{bakhtiari2016parallel} developed a second-order, unconditionally stable semi-Lagrangian scheme with a spatially adaptive Chebyshev octree. \citet{liu2025semi} developed a semi-Lagrangian radial basis function partition of unity closest point method for solving AD equations on surfaces. \textbf{Moving Mesh Methods.} \citet{bergmann2022second} presented a second-order accurate space-time finite volume scheme for solving the AD equation on moving Chimera grids. \citet{sultan2023stable} developed a moving FD meshing scheme for solving linear and nonlinear AD equations. \citet{terekhov2023dynamic} developed a moving mesh FVM for simulating blood flow and coagulation. \textbf{Multiscale Methods.} 
\citet{lee2016multiscale} utilized Fourier decompositions for multiscale velocity fields. \citet{le2017numerical} explored a multiscale FEM (MsFEM) for heterogeneous media with strong advection. \citet{wang2024partitioned} proposed two partitioned schemes based on the variational multiscale method to solve the coupled system modeling solute transport in blood and the surrounding arterial wall. \textbf{Stabilization Techniques.} \citet{le2016stabilization} utilized invariant measures for non-coercive problems.
\citet{benedetto2016order} investigated stabilization in virtual element methods.
\citet{gonzalez2017w} stabilized explicit Runge-Kutta methods via W-methods.
\citet{biezemans2025msfem} improved the MsFEM for accurately solving AD equations with highly oscillatory coefficients (heterogeneous media) and potentially dominant advection. \textbf{Spectral Methods.} \citet{cardone2017exponentially} combined exponential fitting for spatial discretization with an adapted implicit-explicit (IMEX) method for time integration to solve AD problems. \citet{laouar2023efficient} combined a spectral Galerkin approach with Legendre polynomials as basis functions and Gauss-Lobatto quadrature for the AD equation with constant and variable coefficients under mixed Robin boundary conditions. \citet{takagi2024implementation} investigated the possibility of using quantum annealers for flow computations by solving the one-dimensional (1D) AD equation using spectral methods. \textbf{B-splines.} \citet{korkmaz2018numerical} proposed a method that combined trigonometric cubic B-splines for spatial discretization and the Rosenbrock implicit method for time integration. \citet{jena2021computational} applied an octic B-spline collocation approach combined with a Crank-Nicolson scheme. \citet{palav2025redefined} presented a collocation method using redefined fourth-order uniform hyperbolic polynomial B-splines to solve the AD equation. \textbf{Differential Quadrature Method.} \citet{gharehbaghi2016explicit} applied differential quadrature methods for semi-infinite domains. \textbf{Meshless Methods.} \citet{ilati2016remediation} applied meshless techniques for groundwater contamination. \citet{suchde2017flux} introduced a modification to meshfree generalized FDMs to include local balances, resulting in an approximate conservation of numerical fluxes. \citet{askari2017numerical} proposed a meshless method of lines to solve the 1D AD equation, where radial basis functions were used for spatial discretization, and the resulting system of ODEs was solved using the fourth-order Runge-Kutta method. \textbf{DG Methods.} 
\citet{zhang2016operator} integrated positivity-preserving DG with operator splitting for stiff problems. \citet{fiengo2017discontinuous} adapted DG for sediment transport in river networks. \citet{wang2024discontinuous} developed and analyzed a class of primal DG methods for magnetic AD problems. \textbf{Reduced Basis Methods.} \citet{dal2017numerical} analyzed multi-space reduced basis preconditioning for FEM systems. \textbf{Lattice Boltzmann Methods (LBMs).} 
\citet{karimi2016current} analyzed the limitations of existing LBMs in preserving fundamental properties (maximum principle, non-negativity) of AD solutions. \citet{yan2017two} developed the TRT-LBM for transport phenomena. \textbf{Flux Reconstruction.} \citet{romero2017direct} extended the direct flux reconstruction scheme to triangular grids for stability in AD. \textbf{Data Assimilation Methods.} \citet{penenko2017direct} presented a direct variational data assimilation algorithm for the 1D AD model. The algorithm aimed to improve the accuracy of model predictions by adjusting the uncertainty function based on in situ measurements. \textbf{Analytical Solutions.} \citet{lewis2017extension} expanded Stefan-type solutions for complex systems.
\citet{ho2024recursive} proposed a semi-analytical model to simulate the multispecies transport of decaying contaminants, where recursive analytical solutions (Laplace transform and generalized integral transform) were used for efficient computation. 

It is crucial to highlight that diffusion-dominated transport is central to many physical phenomena, from heat dissipation in electronics to the dispersion of pollutants in stagnant water bodies. While numerous numerical methods exist for the AD equation, approaches tailored for accuracy and efficiency in diffusion-dominated regimes remain essential. The FGIG method, as we will demonstrate, is particularly well-suited for such cases due to its inherent stability and ability to minimize temporal error accumulation. This study introduces a novel Fourier-Gegenbauer Integral-Galerkin (FGIG) method for solving the 1D AD equation with periodic boundary conditions. The FGIG method combines the strengths of both Fourier series and Gegenbauer polynomial expansions to achieve high accuracy, efficiency, and stability. In particular, the FGIM employs a Fourier basis to approximate the spatially periodic solution and Gegenbauer polynomials for accurate temporal integration to effectively capture the dynamics of advection and diffusion processes for relatively low P\'{e}clet numbers. Unlike traditional Galerkin methods that yield systems of ordinary differential equations requiring time integration with initial conditions, the FGIG method leads to a system of integral equations that can be solved directly without imposing initial or boundary conditions. Some of the distinctive features offered by the proposed method include: (i) The absence of time-stepping procedures allows for minimizing the accumulation of temporal errors and improving accuracy, (ii) avoiding iterative time integration schemes, by directly solving the resulting system of integral equations, can potentially reduce the computational cost, and (iii) the use of a barycentric formulation of the shifted-Gegenbauer-Gauss (SGG) quadrature, known for its stability and efficiency in evaluating definite integrals, contribute to the accurate and stable computation of the solution and its spatial derivative. 

This work makes the following key contributions: (i) Proposes a novel FGIG method combining Fourier series and Gegenbauer polynomials to solve the 1D AD equation with periodic boundary conditions, (ii) derives a semi-analytical solution using the FGIG framework, offering a very accurate alternative for specific cases, (iii) demonstrates exponential convergence in both space and time for sufficiently smooth solutions, outperforming traditional methods like FD and FVMs with polynomial convergence rates, (iv) ensures stability for relatively low P\'{e}clet numbers, particularly for Gegenbauer parameter values in the range $[0, 0.5]$, even under highly oscillatory conditions, with condition numbers comparable to those of the integration matrix used for time discretization, (v) the FGIM is inherently parallelizable, as its formulation involves independent systems of integral equations. This feature allows for efficient computation on multi-core processors, and significantly reduces the computational time for large-scale problems, unlike traditional methods, which may require sequential computations, (vi) unlike FD, finite element, and semi-Lagrangian methods, which often rely on iterative time-stepping schemes, the FGIM directly solves the problem using integral equations. This eliminates the accumulation of temporal errors, (vii) the barycentric formulation of the SGG quadrature boosts the accuracy and efficiency in computing definite integrals, a step forward compared to conventional integration techniques, (viii) the shift from a differential-equation framework to an integral-equation formulation eliminates the dependency on initial and boundary conditions during the computations and possibly offer a new way to handle similar problems in the field, and furthermore, (ix) the FGIM framework allows for the derivation of semi-analytical solutions. This feature provides a highly accurate alternative for verifying the numerical results, a capability not generally available in many competing methods. To the best of our knowledge, this work presents the first numerical method for solving the 1D AD equation with periodic boundary conditions that uniquely combines the Fourier series for spatial periodicity and Gegenbauer polynomials for temporal integration within a Galerkin framework. While Gegenbauer polynomials and their closely related shifted ultraspherical polynomials have been employed in other contexts, such as solving variable-order fractional PDEs \cite{kumar2019gegenbauer}, fractional AD equations with generalized Caputo derivatives \cite{nagy2024accurate}, and enhancing Runge-Kutta stability \cite{o2019runge}, their integration into a Fourier-Galerkin framework for solving the 1D AD equation represents a novel and impactful contribution.

The remainder of this paper is structured as follows. Section \ref{sec:PN} introduces the sets of symbols and notations used throughout the paper to represent complex mathematical formulas and expressions concisely. These notations generally follow the writing convention established in \cite{Elgindy2023a,elgindy2024numerical,elgindy2024optimal}. Section \ref{sec:PS} introduces the problem under study. Section \ref{sec:FGIG1} presents the FGIG method, describing its formulation and implementation. In Section \ref{sec:CCC1}, we analyze the computational cost and time complexity of the FGIG method. Section \ref{sec:ESA1} presents a rigorous analysis of the error and stability characteristics of the method. Section \ref{sec:FGIG12} derives a semi-analytical solution within the FGIG framework, offering a highly accurate alternative for specific cases. Section \ref{sec:CRAC1} provides computational results that validate the accuracy and efficiency of the FGIG method. Sections \ref{sec:Conc} and \ref{sec:LCM1} conclude the paper with a summary of key findings, a brief discussion of the limitations of the FGIG method, and potential future research directions.  Mathematical proofs are presented in Section \ref{sec:MP1}.
\label{sec:MP1}

\section{Mathematical Notations}
\label{sec:PN}
\noindent\textbf{Logical Symbols.} The  symbols $\forall, \foralla, \forallaa, \foralle, \foralls$, and $\foralll$ stand for the phrases ``for all,'' ``for almost all,'' ``for any,'' ``for each,'' ``for some,'' and ``for a relatively large'' in respective order. $\Re(.)$ and $\Im(.)$ denote the real and imaginary parts of a complex number, respectively.\\[0.5em]
\textbf{List and Set Notations.} $\MFC, \MFF$, and $\MFP$ denote the set of all complex-valued, real-valued, and piecewise continuous real-valued functions, respectively. Moreover, $\MBR$, $\MBRzerP$, $\MBZ, \MBZP, \MBZzerP, \MBZOP$, $\MBZeP$, and $\MBZzereP$ denote the sets of real numbers, nonnegative real numbers, integers, positive integers, non-negative integers, positive odd integers, positive even integers, and non-negative even integers, respectively. The notations $i$:$j$:$k$ or $i(j)k$ indicate a list of numbers from $i$ to $k$ with increment $j$ between numbers, unless the increment equals one where we use the simplified notation $i$:$k$. The list of symbols $y_1, y_2, \ldots, y_n$ is denoted by $\left. y_i \right|_{i=1:n}$ or simply $y_{1:n}$, and their set is represented by $\{y_{1:n}\}\,\foralla n \in \MBZP$ and variable $y$. The list of ordered pairs $(x_1,y_1), (x_1,y_2), \ldots,$ $(x_1, y_n), (x_2,y_1), \ldots,$ $(x_m,y_n)$ is denoted by $(x_{1:m}, y_{1:n})$ and their set is denoted by $\{(x_{1:m}, y_{1:n})\}$. We define $\MBJ_n = \{0:n-1\}, \MBJP_n = \MBJ_n \cup \{n\}$, and $\MBN_n = \{1:n\}\,\foralla n \in \MBZP$. $\MBS_n^{T} = \left\{t_{n,0:n-1}\right\}$ is the set of $n$ equally-spaced points such that $t_{n,j} = T j/n\, \forall j \in \MBJ_n$. $\MBK_n = \{-n/2:n/2\}, \MBK'_n = \MBK\backslash\{n/2\}, \MBK''_n = \MBK\backslash\{\pm n/2\}$, and $\MBKcanc{n} = \MBK_n\backslash\{0\} \foralla n \in \MBZeP$. $\MBG_n^{\lambda} = \left\{z_{n,0:n}^{\lambda}\right\}$ is the set of Gegenbauer-Gauss (GG) zeros of the $(n+1)$st-degree Gegenbauer polynomial with index $\lambda > -1/2$, and $\MBSG_{L,n}^{\lambda} = \left\{\hz_{n,0:n}^{\lambda}: \hz_{n,0:n}^{\lambda} = L \left(z_{n,0:n}^{\lambda}+1\right)/2\right\}$ is the SGG points set in the interval $[0,L]\,\foralla n \in \MBZP$; cf. \cite{Elgindy201382,elgindy2018high,elgindy2018optimal}. We also define the augmented SGG (ASGG) points set $\MBSG_{L,n}^{\lambda, +} = \MBSG_{L,n}^{\lambda} \cup \{L\}$. $\FOmega_{a,b}$ denotes the closed interval $[a, b] \foralla a, b \in \MBR$. $\IFOmega$ is the interior of the set $\FOmega$. The specific interval $[0, T]$ is denoted by $\FOmega_{T}\,\forall T > 0$. The Cartesian product $\FOmega_{L} \times \FOmega_{T} = \{(x,t): x \in \FOmega_L, t \in \FOmega_T\}$ is denoted by $\FOmega_{L \times T} \foralla L, T \in \MBR^+$. $\bm{\beta} = [-\beta, \beta]\,\forall \beta > 0$, and ${\F{C}_{T,\beta} } = \left\{ {x + iy:x \in {\FOmega_{T}},y \in \bm{\beta}} \right\}\;\forall \beta  > 0$.\\[0.5em]
\textbf{Special Mathematical Constant.} $u_R$ is the unit round-off of the floating-point system; typically $u_R \sim 10^{-16}$ for double-precision.\\[0.5em]
\textbf{Function Notations.} $\delta_{n,m}$ is the usual Kronecker delta function of variables $n$ and $m$. $\Gamma$ denotes the Gamma function. $g^*$ denotes the complex conjugate of $g \foralla g \in \MFC$. For convenience, we shall denote $g(t_{n})$ by $g_n \foralla t_n \in \MBR, n \in \MBZ$, unless stated otherwise. Similarly, for a function with a subscript, such as $g_N$, we shall denote $g_N(t_n)$ by $g_{N,n}$. We extend this writing convention to multidimensional functions; for example, to evaluate a bivariate function $u$ at some discrete points set $\left\{(x_{0:m-1},t_{0:n-1})\right\}$, we write $u_{l,j}$ to simply mean $u(x_l,t_j)\,\forall (l,j) \in \MBJ_{m \times n}$. $u_{0:m-1,0:n-1}$ stands for the list of function values $u_{0,0}, u_{0,1}, \ldots, u_{0,n-1}, u_{1,0}, \ldots, u_{1,n-1}, \ldots,$ $u_{m-1,n-1}$.\\[0.5em]
\textbf{Integral Notations.} We denote $\int_0^{b} {h(t)\,dt}$ and $\int_a^{b} {h(t)\,dt}$ by $\C{I}_{b}^{(t)}h$ and $\C{I}_{a, b}^{(t)}h$, respectively, $\foralla$ integrable $h \in \MFC, a, b \in \MBR$. If the same argument occurs in both the subscript and the superscript, the definite integral is performed with respect to (w.r.t.) any dummy integration variable. For instance, the notation $\C{I}_t^{(t)}h$ stands for $\int_0^t {h(.)\,d(.)} \foralla$ dummy integration variable. If the integrand function $h$ is to be evaluated at any other expression of $t$, say $u(t)$, we express $\int_0^{b} {h(u(t))\,dt}$ and $\int_a^b {h(u(t))\,dt}$ with a stroke through the square brackets as $\C{I}_{b}^{(t)}h\cancbra{u(t)}$ and $\C{I}_{a,b}^{(t)}h\cancbra{u(t)}$ in respective order. The notation $\C{I}_{\FOmega_{a,b}}^{(x)} h$ simply means the definite integral $\int_a^b {h(x)\,dx}\,\foralla \FOmega_{a,b} \subseteq \MBR$. One can easily extend all of these notations to multivariate functions. For example, for a bivariate function $h(x,y)$, the notation 
$\C{I}_y^{(y)}h$ stands for $\int_0^y {h(x,.)\,d(.)} \foralla$ dummy integration variable.\\[0.5em] 
\textbf{Inner Product Notations.} Given $f, g \in \MFC$ and defined on $\FOmega_{a,b}$, we denote their weighted inner product w.r.t. the weight function $w$ by $(f,g)_w$, and define it as $(f,g)_w = \C{I}_{a,b}^{(x)} (f g^* w)$, where $w \in \MFF$ is a non-negative, integrable function defined on $\FOmega_{a,b}$. The inner product $(f,g)_{n,w}$ stands for the discrete inner product that is a discrete approximation of the continuous inner product $(f,g)_w$, calculated using $n$ points, $\foralla n \in \MBZP$. The subscript $w$ is omitted from either inner product notations when $w(x) = 1 \forall x \in \FOmega_{a, b}$.\\[0.5em] 
\textbf{Space and Norm Notations.} $\MBT_{T}$ is the space of $T$-periodic, univariate functions $\foralla T \in \MBRP$. $\Def{\FOmega}$ is the space of all functions
defined on the set $\FOmega$. $C^k(\FOmega)$ is the space of $k$ times continuously differentiable functions on ${\FOmega}\,\forall k \in \MBZzerP$. $L^p({\FOmega})$ is the Banach space of measurable functions $u \in \Def{\FOmega}$ such that ${\left\| u \right\|_{{L^p}}} = {\left( {{\C{I}_{\FOmega}}{{\left| u \right|}^p}} \right)^{1/p}} < \infty\,\forall 1 \le p < \infty$. The $L^p$ space on the rectangle $\FOmega_{L \times T}$ is denoted by $L^p(\FOmega_{L\times T})$ and consists of all scalar measurable functions $u \in \Def{\FOmega_{L \times T}}$ for which ${\left\| u \right\|_{{L^p}}} = {\left( {{\C{I}_{\FOmega_L}} {\C{I}_{\FOmega_T}}{{\left| u \right|}^p}} \right)^{1/p}} < \infty\,\forall 1 \le p < \infty$. The space 
\[{H^s}({\FOmega_{T}}) = \left\{ {u \in {L_{loc}}({\FOmega_{T}}),\;{D^\alpha }u \in {L^2}({\FOmega_{T}})\;\forall \left| \alpha  \right| \le s\;} \right\}\;\forall s \in \MBZzerP,\] 
is the inner product space with the inner product 
\[{(u,v)_s} = \sum\nolimits_{\left| \alpha  \right| \le s} {\C{I}_{{\FOmega_{T}}} ^{(t)}\left( {{D^\alpha }u\,{D^\alpha }v} \right)},\]
where ${{L_{loc}}({\FOmega_{T}} )}$ is the space of locally integrable functions on ${\FOmega_{T}}$ and ${{D^\alpha }u}$ denotes any derivative of $u$ with multi-index $\alpha$. Moreover,
\[\C{H}_{T}^s = \left\{ {u \in {H^s}({\FOmega_{T}}),\;{u^{(s)}} \in {BV},\;{u^{(0:s - 1)}}(0) = {u^{(0:s - 1)}}(T)} \right\},\]
where $u^{(0:s-1)}$ denotes the column vector of derivatives $[u, u'$, $\ldots, u^{(s-1)}]^{\top}$, and $\displaystyle{{BV} = \left\{ {u \in {L^1}({\FOmega_{T}}):{{\left\| u \right\|}_{BV}} < \infty } \right\}}$ with the norm ${{\left\| u \right\|}_{BV}} = \sup \left\{ {\C{I}_{T}^{(x)}(u\phi '),\;\phi  \in \C{D}({\FOmega_{T}}),\;{{\left\| \phi  \right\|}_{{L^\infty }}} \le 1} \right\}$ such that 
\[\C{D}({\FOmega_{T}}) = \left\{u \in {C^\infty }({\FOmega_{T}}):{\text{supp}}(u)\text{ is a compact subset }\text{ of }{\FOmega_{T}} \right\}.\] We define the Sobolev space $\ME{H}^s(\FOmega) = \{u \in L^2(\FOmega): D^{\alpha} u \in L^2(\FOmega)\;\forall |\alpha| \le s\} \foralla$ open set $\FOmega \subset \MBR$. This space is a Hilbert space with the inner product
\[{\left( {f,g} \right)^{s,\FOmega}} = \sum\limits_{\left| \alpha  \right| \le s} {\left( {{D^\alpha }f,{D^\alpha }g} \right)}  = \sum\limits_{\left| \alpha  \right| \le s} {\C{I}_{\FOmega}^{(x)}[{D^\alpha }f\,{{({D^\alpha }g)}^*}]}.\]
For analytic functions, we define the space
\[\C{A}_{T,\beta} = \{v \in \C{H}_{T}^{\infty}: v \text{ is analytic in some open set containing }\F{C}_{T,\beta}\},\]
with the norm ${\left\| v \right\|_{{\C{A}_{T,\beta}}}} = {\left\| v \right\|_{{L^\infty }({{\F{C}}_{T,\beta}})}}$. For convenience of writing, we shall denote ${\left\|  \cdot  \right\|_{{L^2}({\FOmega_{T}})}}$ by $\left\|  \cdot  \right\|$, and call a function $v \in \C{A}_{T,\beta}$ \textit{``a $\beta$-analytic function''} if $v$ is analytic on ${\F{C}_{T,\infty}}$ and $\displaystyle{{\lim _{\beta  \to \infty }}\frac{{{{\left\| v \right\|}_{{\C{A}_{T,\beta} }}}}}{{{e^{{\omega _\beta }}}}} = 0}$, where ${\omega _{\gamma}} = 2\pi \gamma/T\,\forall \gamma \in \MBR$. $L^{\infty}({\FOmega})$ is the space of all essentially bounded measurable functions defined on a measurable set $\FOmega$ with $L^{\infty}$ norm defined as $\left\|f\right\|_{L^{\infty}(\FOmega)} = \sup |f(x)| = \inf\{M \ge 0: |f(x)| \le M\,\forallaa x \in \FOmega\}$. $\left\|\cdot\right\|_1$ and $\left\|\cdot\right\|_2$ denote the usual $l_1$-norm and Euclidean norm of vectors, respectively.\\[0.5em]
\textbf{Vector Notations.} We shall use the shorthand notations $\bmt_N$ and $g_{0:N-1}$ to stand for the column vectors $[t_{N,0}, t_{N,1}, \ldots$, $t_{N,N-1}]^{\top}$ and $[g_0, g_1, \ldots, g_{N-1}]^{\top}\,\forall N \in \MBZP, g \in \MFC$, and any variable $t$ in respective order. The row vector $[t_{N,0}, t_{N,1}, \ldots$, $t_{N,N-1}]$ is denoted by $\bmt_N^{\top}$ or $[t_{N,0:N-1}]$. $\Sigma {g_{0:N-1}}$ is the sum of the elements of the vector $g_{0:N-1}$. In general, $\foralla h \in \MFC$ and vector $\bmy$ whose $i$th-element is $y_i \in \MBR$, the notation $h(\bmy)$ stands for a vector of the same size and structure of $\bmy$ such that $h(y_i)$ is the $i$th element of $h(\bmy)$. Moreover, by $\bmh(\bmy)$ or $h_{1:m}\cancbra{\bmy}$ with a stroke through the square brackets, we mean $[h_1(\bmy), \ldots, h_m(\bmy)]^{\top}\,\foralla m$-dimensional column vector function $\bmh$, with the realization that the definition of each array $h_i(\bmy)$ follows the former notation rule $\foralle i$. $\Sigma h_{1:m}\cancbra{\bmy}$ is the sum of the elements of the vector function $h_{1:m}\cancbra{\bmy}$. Furthermore, if $\bmy$ is a vector function, say $\bmy = \bmy(t)$, then we write $h(\bmy(\bmt_N))$ and $\bmh(\bmy(\bmt_N))$ to denote $[h(\bmy(t_0)), h(\bmy(t_1)), \ldots, h(\bmy(t_{N-1}))]^{\top}$ and $[\bmh(\bmy(t_0)), \bmh(\bmy(t_1)), \ldots, \bmh(\bmy(t_{N-1}))]^{\top}$ in respective order.\\[0.5em] 
\textbf{Matrix Notations.} $\F{O}_n, \F{1}_n$, and $\F{I}_n$ stand for the zero, all ones, and the identity matrices of size $n$. $\F{C}_{n,m}$ indicates that $\F{C}$ is a rectangular matrix of size $n \times m$; moreover, $\F{C}_n$ denotes a row vector whose elements are the $n$th-row elements of $\F{C}$, except when $\F{C}_n = \F{O}_n, \F{1}_n$, or $\F{I}_n$, where it denotes the size of the matrix. For convenience, a vector is represented in print by a bold italicized symbol while a two-dimensional matrix is represented by a bold symbol, except for a row vector whose elements form a certain row of a matrix where we represent it in bold symbol as stated earlier. For example, $\bmone_n$ and $\bmzer_n$ denote the $n$-dimensional all ones- and zeros- column vectors, while $\F{1}_n$ and $\F{O}_n$ denote the all ones- and zeros- matrices of size $n$, respectively. \Kappa$(\F{A})$ denotes the condition number of a square matrix $\F{A}$. $\resh{m}{n}{\F{A}}$ and $\reshs{n}{\F{A}}$ are the matrices obtained by reshaping $\F{A}$ into an $m$-by-$n$ matrix and a square matrix of size $n$, respectively while preserving their column-wise ordering from $\F{A}$. Finally, the notations $\otimes, \odot$, and $[.;.]$ represent the Kronecker product, Hadamard product, and the usual vertical concatenation, respectively.

\section{Problem Statement}
\label{sec:PS}
In this study, we focus on the following 1D AD equation with periodic boundary conditions:
\begin{subequations}
\begin{equation}\label{eq:AD1}
{u_t} + \mu {u_x} = \nu {u_{xx}},\quad \forall (x,t) \in \FOmega_{L\times T},\quad \foralls L, T \in \MBR^+,
\end{equation}
with the given initial condition
\begin{equation}\label{eq:AD2}
u(x,0) = u_0(x),\quad \forall x \in \FOmega_L,
\end{equation}
and the periodic Dirichlet boundary conditions
\begin{equation}\label{eq:AD3}
u(x+L,t) =  u(x,t),\quad \forall t \in \FOmega_T,\quad x \in \MBRzerP.
\end{equation}
\end{subequations}
The nonnegative constants $\mu$ and $\nu$ denote the advection velocity and diffusion coefficient, respectively. Eqs. \eqref{eq:AD1}-\eqref{eq:AD3} describe the initial-boundary value problem of the AD equation in strong form, which we refer to by Problem S. The solution to the problem, $u$, represents the concentration or distribution of a scalar quantity (like heat, pollutants, or particles) in a medium over space and time, taking into account both advection and diffusion processes. Problem S serves as a valuable benchmark problem for testing the accuracy, stability, and efficiency of numerical methods. It also acts as a prototype for developing more sophisticated numerical techniques that can be applied to more complex problems in higher dimensions and with more realistic boundary conditions.

\section{The FGIG Method}
\label{sec:FGIG1}
The classical solution of Problem $S$ must be at least twice differentiable in space. To allow a larger class of solutions, we rewrite the time-integrated formulation of the PDE in a weak form. To this end, let us partially integrate both sides of the PDE w.r.t. time on $\FOmega_t$ and impose the initial condition \eqref{eq:AD2} to obtain \textit{``the time-integrated (TI) form of the PDE''} as follows:
\begin{equation}
u + \C{I}_t^{(t)}(\mu {u_x} - \nu {u_{xx}}) = {u_0}.
\end{equation}
Multiplying both sides of the equation by an arbitrary function $\varphi \in \MBT_L \cap \ME{H}^1(\IFOmega_L)$ and partially integrating over $\FOmega_L$ yields
\begin{equation}\label{eq:WFTIPDE}
\left( {u + \C{I}_t^{(t)}(\mu {u_x} - \nu {u_{xx}}),\varphi } \right) = \left( {{u_0},\varphi } \right),
\end{equation}
which represents \textit{``the weak TI (WTI) form of the PDE.''} Integration by parts gives
\begin{equation}\label{eq:IBP1}
\left( {u + \mu \C{I}_t^{(t)}{u_x},\varphi } \right) + \left( {\nu \C{I}_t^{(t)}{u_x},{\varphi _x}} \right) = \left( {{u_0},\varphi } \right) + \left. {\nu \varphi \C{I}_t^{(t)}{u_x}} \right|_0^L.
\end{equation}
The last term vanishes due to the spatial periodicity of $u$ and $\varphi$; therefore
\begin{equation}\label{eq:IBP1Koky1}
\left( {u + \mu \C{I}_t^{(t)}{u_x},\varphi } \right) + \left( {\nu \C{I}_t^{(t)}{u_x},{\varphi _x}} \right) = \left( {{u_0},\varphi } \right).
\end{equation}
Now, let $u_x = \psi \foralls \psi \in L^2(\FOmega_{L \times T}), u(0,t) =  u(L,t) = g(t)\,\forall t \in \FOmega_T$, and impose the left boundary condition to obtain
\begin{equation}\label{eq:AS1}
u = \C{I}_x^{(x)} \psi + g.
\end{equation} 
Substituting Eq. \eqref{eq:AS1} into Eq. \eqref{eq:IBP1Koky1} yields:
\begin{equation}\label{eq:FF1}
\left( {\left[ {\C{I}_x^{(x)} + \mu \C{I}_t^{(t)}} \right]\psi ,\varphi } \right) + \left( {\nu \C{I}_t^{(t)}\psi ,{\varphi _x}} \right) = \left( {{u_0} - g,\varphi } \right).
\end{equation}
We refer to this new form by \textit{``the spatial-derivative-substituted (SDS) WTI form of the PDE,''} and call its solution, $\psi$,\textit{``the solution spatial derivative (SSD) of Problem S.''} Notice how the SDS-WTI form of the PDE requires only that the weak solution for Problem S be square integrable, as shown by Theorem \ref{thm:1}, which is much less restrictive than requiring twice differentiability of $u$. 

The spatial periodic nature of $u$ allows us to approximate \textit{``the time-dependent offset solution (TDOS)''}, $\C{I}_x^{(x)} \psi$, by a truncated Fourier basis expansion. In particular, let us consider the $N/2$-degree, $L$-spatially-periodic (sp), truncated Fourier series with time-dependent coefficients, ${}_N\C{I}_x^{(x)}\psi$, written in the following modal form:
\begin{equation}\label{eq:FI1nn1}
{}_N\C{I}_x^{(x)} \psi(x,t) = \sum\limits_{k \in \MBK_N} {{\tilpsi_k(t)}\; {e^{i{\omega _k}x}}},\quad \foralls N \in \MBZeP.
\end{equation}
Notice that there are $N + 1$ unknown coefficient functions for us to determine. Fortunately, since 
\begin{equation}\label{eq:NKR1}
\C{I}_0^{(x)} \psi = 0 = \Sigma {{\tilpsi_{-N/2:N/2}\cancbra{t}}},
\end{equation}
by definition, we actually need to find $N$ unknown coefficients, as the last coefficient is automatically computed via Eq. \eqref{eq:NKR1}. In this work, we seek all coefficients with nonzero indices first, and then easily calculate $\tilpsi_0$ by using the following formula:
\begin{equation}\label{eq:Smart1}
\tilpsi_0(t) = -\sum\limits_{k \in \MBKcanc{N}} {{\tilpsi_k(t)}}.
\end{equation}
To determine the unknown coefficients, we require that the approximation ${I_N}\psi$ satisfies the SDS-WTI form of the PDE \eqref{eq:FF1} for $\varphi = e^{i \omega_n x}\;\foralle n \in \MBKcanc{N}$. That is, the spectral approximate SSD is the one that satisfies
\begin{equation}
\left( {\left[ {\C{I}_x^{(x)} + \mu \C{I}_t^{(t)}} \right]\psi ,{e^{i{\omega _n}x}}} \right) + \left( {\nu \C{I}_t^{(t)}\psi,\partial_x {e^{i{\omega _n}x}}} \right) = \left( {{u_0} - g,{e^{i{\omega _n}x}} } \right),
\end{equation}
$\forall n \in \MBKcanc{N}$. This can be reduced further into the following form:
\begin{equation}\label{eq:FF3}
\left( {\left[ {\C{I}_x^{(x)} + (\mu-i \omega_n \nu) \C{I}_t^{(t)}} \right]\psi ,{e^{i{\omega _n}x}}} \right) = \left( {{u_0},{e^{i{\omega _n}x}} } \right),\quad \forall n \in \MBKcanc{N},
\end{equation}
by realizing that $\left( {g,{e^{i{\omega _n}x}} } \right) = 0\,\forall n \in \MBKcanc{N}$. Using the Fourier quadrature rule based on the equally-spaced nodes set $\MBS_N^{L} = \{{x_{N,0:N-1}}\}$: 
\[{Q_F}(f) = \frac{L}{N} \Sigma{f_{0:N-1}},\quad \forall f \in \MBT_T,\]
we can define the following discrete inner product
\[{(u,v)_N} = \frac{L}{N} \left({{u_{0:N-1}^{\top}}v_{0:N-1}^*}\right),\quad\foralla u, v \in \MFC.\]
Consider now the $N/2$-degree, $L$-periodic Fourier interpolant ${I_N}u_0$ that matches $u_0$ at the set of nodes $\MBS_N^L$ so that
\begin{equation}\label{eq:FI1}
{I_N}u_0(x) = \sum\limits_{k \in \MBK_N} {\frac{{{{\tilu_{0,k}}}}}{{{c_k}}}{e^{i{\omega _k}x}}},
\end{equation}
where 
\begin{empheq}[left={c_k = }\empheqbiglbrace]{align*}
  1, &\quad k \in \MB{K}''_N,\\
  2, &\quad k =  \pm N/2,
\end{empheq}
and ${\tilu_{0,k}}$ is the discrete Fourier transform (DFT) interpolation coefficient given by
\[{\tilu_{0,k}} = \frac{1}{L}{\left( {u_0,{e^{i \omega_k x}}} \right)_N} = \frac{1}{N}\sum\limits_{j \in \MBJ_N} {{u_{0,j}}{e^{ - i \homega_{k j,N}}}},\quad \forall k \in \MB{K}_N,\]
with $\tilu_{0,N/2} = \tilu_{0,-N/2}$ and $\homega_{k,N} = 2 \pi k/N$; cf. \cite{Elgindy2023a}. Since $u_0$ is periodic, the DFT coefficients satisfy the additional conjugate symmetry condition $\tilu_{0,n} = \tilu_{0,-n}^*\,\forall n \in \MBK_N$. This implies that for real-valued, periodic signals, only half of the DFT coefficients need to be computed, as the other half can be obtained through conjugation.

Substituting \eqref{eq:FI1nn1} and \eqref{eq:FI1} into \eqref{eq:FF3} yields the following linear system of integral equations:
\begin{gather}\label{eq:FF3_1}
\sum\limits_{k \in \MBKcanc{N}} {\left[ {\tilpsi_k(t) + \omega_k (\nu \omega_n + \mu i) \C{I}_t^{(t)} \tilpsi_k} \right] \left(e^{i \omega_k x}, e^{i \omega_n x}\right)}\\ = \sum\limits_{k \in \MBKcanc{N}} {\frac{\tilu_{0,k}}{c_k} \left(e^{i \omega_k x}, e^{i \omega_n x}\right)},\quad \forall n \in \MBKcanc{N}.
\end{gather}
The orthogonality property of complex exponentials 
\begin{equation}
\left({e^{i{\omega _k}x}},{e^{i{\omega _n}x}}\right) = L\;\delta_{k,n},\quad \forall \{k,n\} \subset \MBK_N,
\end{equation}
allows us to reduce System \eqref{eq:FF3_1} into the following linear system of integral equations:
\begin{equation}\label{eq:nnmm2}
\ME{A}_n^{(t)}\;{\tilpsi}_n(t) = \gamma_n,\quad \forall n \in \MBKcanc{N},
\end{equation}
where 
\begin{subequations}
	\begin{gather}
\ME{A}_n^{(t)} = 1 + \alpha_n \C{I}_t^{(t)},\quad \alpha_n = \omega_n (\nu \omega_n + \mu i),\\
\gamma_n = \frac{\tilu_{0,n}}{c_n},\quad \forall n \in \MBKcanc{N}.
\end{gather}
\end{subequations}

Before we continue our prescription of the proposed method, it is interesting to note here that unlike the standard Galerkin method, which yields a system of ordinary differential equations requiring initial conditions for time integration, the present approach results in the system of integral equations \eqref{eq:nnmm2} that can be solved without imposing any initial or boundary conditions. The approximate solution, ${}_Nu$, can be directly computed at any point in space and time by adding the boundary function $g$ to the TDOS through Eq. \eqref{eq:AS1}, after reconstructing the latter using Eqs. \eqref{eq:FI1nn1} and \eqref{eq:Smart1}.

To slightly improve the accuracy of the system model \eqref{eq:nnmm2}, we can increase the degree of the Fourier interpolant of $u_0, I_Nu_0$, since $u_0$ is already given, and the first $N$ DFT interpolation coefficients of two Fourier interpolations of degrees $n/2$ and $m/2$ for the same function are generally not the same when $n$ and $m$ are distinct $\foralla n, m \in \MBZeP: N \le \min\{n,m\}$. This is because the coefficients in DFT interpolation depend on the number of sample points (or the degree of the interpolation) and how they are distributed, which changes with $n$ and $m$. Generally, as the degree of interpolation increases, the approximation tends to capture more details of the function. This means that the discrete coefficients for a higher degree interpolation are often more accurate in representing the true Fourier coefficients of the function, especially for smooth functions. With this observation, we can accurately re-represent $u_0$ by an $N_0/2$-degree, $L$-periodic Fourier interpolant, ${I_{N_0}}u_0$, that matches $u_0$ at the set of nodes $\MBS_{N_0}^L \foralls N_0 \in \MBZeP: N_0 > N$. Using \cite[Eq. (4.2)]{Elgindy2023a}, we can write
\begin{equation}\label{eq:FI1newu01}
{I_{N_0}}u_0(x) = \sumd\sum\limits_{k \in \MBK_{N_0}} {{\hu_{0,k}}\,{e^{i{\omega _k}x}}},
\end{equation}
where the primed sigma denotes a summation in which the last term is omitted, and ${\hu_{0,k}}$ is the DFT interpolation coefficient given by
\begin{equation}\label{eq:dfjvdhkb1}
{\hu_{0,k}} = \frac{1}{L}{\left( {u_0,{e^{i \omega_k x}}} \right)_{N_0}} = \frac{1}{N_0}\sum\limits_{j \in \MBJ_{N_0}} {{u_{0,j}}{e^{ - i \homega_{k j,N_0}}}},\quad \forall k \in \MBK'_{N_0}.
\end{equation}
By following the same procedure prescribed earlier in this section, we can easily show that the linear system of integral equations \eqref{eq:nnmm2} can be replaced with the more accurate model:
\begin{equation}\label{eq:nnmm3}
\ME{A}_n^{(t)}\;{\tilpsi}_n(t) = \hu_{0,n},\quad \forall n \in \MBKcanc{N}.
\end{equation}

Now, we shift gears, and turn our attention into how we can effectively solve system \eqref{eq:nnmm3} numerically. Notice first that while complex exponentials are excellent as basis functions for representing periodic functions, they are often less suitable for nonperiodic functions. On the other hand, orthogonal polynomials are specifically designed to approximate nonperiodic functions over finite intervals, making them a more natural choice for nonperiodic integral equations. Since the coefficient vector function ${\tilpsi_{-N/2:N/2}}\cancbra{t}$ is generally nonperiodic, we solve System \eqref{eq:nnmm3} using a variant of the shifted-Gegenbauer (SG) integral pseudospectral (SGIPS) method of \citet{elgindy2018optimal} and \citet{Elgindy2023a}. The latter methods use the barycentric SGG quadratures, constructed using the stable barycentric representation of shifted Lagrange interpolating polynomials and the explicit barycentric weights for the SGG points, well known for their stability, efficiency, and superior accuracy. These barycentric SGG quadratures allow us to construct the SG integration matrices (SGIMs) (aka the SG operational matrices of integration) that can effectively evaluate the sought definite integration approximations using matrix–vector multiplications, often leading into well-conditioned systems of algebraic equations. In the present work, we generate the necessary barycentric SGIMs using \cite[Eq. (4.38)]{Elgindy20161}, which directly defines the required barycentric SGIMs in terms of the barycentric Gegenbauer integration matrices (GIMs), derived in \cite{Elgindy20171}. This modification to the methods of \citet{Elgindy2023a,elgindy2018optimal} skips the need to shift the quadrature nodes and weights and Lagrange polynomials from their original time domain $\FOmega_{-1,1}$ into the shifted time domain $\FOmega_T$, and directly constructs the desired SGIM by premultiplying the usual GIM by a scalar multiple, unless the approximate solution values are required at non-collocation time nodes. That is, if we denote the 1st-order barycentric SGIM by ${}_{T}\F{Q}$, then \cite[Eq. (4.38)]{Elgindy20161} immediately tells us that
\begin{equation}\label{eq:NewSGIM1}
{}_{T}\F{Q} = \frac{T}{2} \F{Q},
\end{equation} 
where $\F{Q}$ is the 1st-order barycentric GIM based on the stable barycentric representation of Lagrange interpolating polynomials and the explicit barycentric weights for the GG points, as described in \citet{Elgindy20171}. 

To collocate System \eqref{eq:nnmm3} in the SG physical space, we first write the system at the SG set of mesh points $\MBSG_{T,M}^{\lambda} = \{\hatt_{M,0:M}^{\lambda}\} \foralls M \in \MBZP$:
\begin{equation}\label{eq:nnmm3Hi1}
{\tilpsi}_{n,j} + \alpha_n \C{I}_{t_j}^{(t)}\;{\tilpsi}_n = \hu_{0,n},\quad \forall j \in \MBJP_M, n \in \MBKcanc{N}.
\end{equation}
Next, we discretize System \eqref{eq:nnmm3Hi1} at the SG mesh set $\MBSG_{T,M}^{\lambda}$ with the aid of Eq. \eqref{eq:NewSGIM1} to obtain the following $N \times (M+1)$ linear systems of algebraic equations:
\begin{equation}\label{eq:nnmm3Hi2}
{\tilpsi}_{n,j} + \alpha_n\;{}_TQ_j \;{\tilpsi}_{n,0:M} = \hu_{0,n},\quad \forall j \in \MBJP_M, n \in \MBKcanc{N},
\end{equation}
which can be written in the following matrix form:
\begin{equation}\label{eq:nnmm3Hi2_2}
{}_T\F{A}^{(n,M)}\;{\tilpsi}_{n,0:M} = \hu_{0,n} \bmone_{M+1},\quad \forall n \in \MBKcanc{N},
\end{equation}
where 
\begin{equation}\label{eq:habiby1}
{}_T\F{A}^{(n,M)} = \F{I}_{M+1} + \alpha_n\;{}_{T}\F{Q},
\end{equation}
and
\[{\tilpsi_{n,0:M}} = \left[ {\begin{array}{*{20}{c}}
\tilpsi_n(\hatt_{M,0}^{\lambda})&\tilpsi_n(\hatt_{M,1}^{\lambda})&\ldots&\tilpsi_n(\hatt_{M,M}^{\lambda})
\end{array}} \right]^{\top},\quad \forall n \in \MBKcanc{N}.\]
Solving System \eqref{eq:nnmm3Hi2_2} provides $N$ Fourier coefficients at the time collocation points. The values of the $(N+1)$st coefficient at the same nodes can be recovered via Eq. \eqref{eq:Smart1} as explained earlier. The solution $u$ can therefore be readily computed at the collocation points by using Eqs. \eqref{eq:AS1} and \eqref{eq:FI1nn1}:
\begin{equation}\label{eq:FI1nn1Dec22}
{}_{N,M}u\left(x_{N,j},\hatt_{M,l}^{\lambda}\right) = {}_N\C{I}_{x_{N,j}}^{(x)} \psi\left(x,\hatt_{M,l}^{\lambda}\right) + g\left(\hatt_{M,l}^{\lambda}\right),
\end{equation}
$\forall \left(x_{N,j},\hatt_{M,l}^{\lambda}\right) \in \MBS_N^L \times \MBSG_{T,M}^{\lambda}$. For faster computations of Formula \eqref{eq:FI1nn1Dec22} compared to running explicit for loops, we can quickly synthesize ${}_{N,M}u$ at the collocation points using the following useful formula, written in matrix form:
\begin{gather}\label{eq:FI1nn1Dec222}
\text{resh}_{M+1,N}(_{N,M}u_{0:N-1,0:M}) = \text{resh}_{M+1,N+1}(\tilpsi_{-N/2:N/2}\cancbra{\bmt_M^{\lambda}})\;{e^{i \omega_{-N/2:N/2} {\bmx_N^{\top}}}}\\
+ g_{0:M} \otimes \bmone_N^{\top},
\end{gather}
where $g_{0:M} = g\left(\bmt_M^{\lambda}\right)$. Moreover, to determine the solution $u$ at any non-collocation temporal points in $\FOmega_{L \times T}$, we can approximate the coefficients' values using the inverse SGG transform, written in the following Lagrange form:
\begin{equation}\label{sec:ort:eq:Lagint1}
	\tilpsi_n(t) \approx \tilpsi_{n,0:M}^{\top}\,\C{L}_{0:M}^{(\lambda)}\cancbra{t},\quad \forall t \notin \MBSG_{T,M}^{\lambda},
\end{equation}
where $\C{L}_{0:M}^{(\lambda)}\cancbra{t}$ is the shifted Lagrange interpolating polynomial vector function in barycentric-form, whose elements functions are defined by
\begin{equation}\label{eq:new1}
\C{L}_{l}^{(\lambda )}(t) = \frac{{\xi _{l}^{(\lambda )}}}{{{t} - \hatt_{{M},l}^{\lambda }}}/\sum\limits_{j \in \MBJP_M} {\frac{{\xi _{j}^{(\lambda )}}}{{{t} - \hatt_{{M},j}^{\lambda }}}},\quad \forall l \in \MBJP_{M},
\end{equation}
and the barycentric weights, $\xi_{0:M}^{(\lambda )}$, associated with the SGG points can be expressed explicitly in terms of the corresponding Christoffel numbers $\varpi _{M,0:M}^{(\lambda )}$ in the following algebraic form:
\begin{equation}\label{eq:baryweight1}
	\xi _{l}^{(\lambda )} = 2{( - 1)^l}\sqrt {{4^\lambda }{T^{ - 2(1 + \lambda )}}\left( {T - \hatt_{M,l}^{\lambda }} \right)\,\hatt_{M,l}^{\lambda }\,\varpi _{M,l}^{(\lambda )}},\quad \forall l \in \MBJP_M,
\end{equation}
or in the following more numerically stable trigonometric form
\begin{equation}\label{eq:newbaryform14May20222}
\xi _{l}^{(\lambda )} = {\left( { - 1} \right)^l}{\left( {\frac{2}{T}} \right)^\lambda }\sin \left( {\cos^{ - 1}\left( {\frac{2\,\hatt_{M,l}^{\lambda }}{T} - 1} \right)} \right)\sqrt {\varpi _{M,l}^{(\lambda )}},\quad \forall l \in \MBJP_M.
\end{equation}
cf. \cite{elgindy2018optimal,Elgindy2023a}. 

It is noteworthy to mention that one key advantage of the proposed FGIG method lies in its ability to not only compute the solution within the desired domain but also accurately determine its spatial derivative as a byproduct of Eq. \eqref{eq:FI1nn1}:
\begin{equation}\label{eq:TSSKR1}
{}_Nu_x(x,t) = {}_N\psi(x,t) \approx i \sum\limits_{k \in \MBK_N} {\omega _k {\tilpsi_k(t)}\; {e^{i{\omega _k}x}}},\quad \forall (x,t) \in \FOmega_{L \times T},
\end{equation}
which offers a comprehensive information about the solution's characteristics. We can swiftly compute ${}_Nu_x$ at the collocation points using the following efficient formula, written in matrix form:
\begin{gather}\label{eq:FI1nn1Dec2223}
\text{resh}_{M+1,N}^{\top}(_{N,M}u_{x,0:N-1,0:M}) = i\;\left((\omega_{-N/2:N/2}^{\top}\otimes\bmone_N) \odot {e^{i {\bmx_N}\omega_{-N/2:N/2}^{\top}}}\right)\\
\times\;\text{resh}_{M+1,N+1}^{\top}(\tilpsi_{-N/2:N/2}\cancbra{\bmt_M^{\lambda}}),
\end{gather}
where $_{N,M}u_{x,j,l} = {}_{N,M}u_x\left(x_{N,j},\hatt_{M,l}^{\lambda}\right)\;\forall \left(x_{N,j},\hatt_{M,l}^{\lambda}\right) \in \MBS_N^L \times \MBSG_{T,M}^{\lambda}$.
The FGIG method's implementation is outlined in the flowchart of Figure \ref{fig:fgig_flowchart}.

Another interesting feature in the proposed method appears in the process of recovering $u$ from $\C{I}_x^{(x)} \psi$, as described by Eq. \eqref{eq:AS1}, which is entirely free from truncation errors, thus maintaining the full precision attained in the calculation of ${}_N\C{I}_x^{(x)} \psi$. 

A third key feature of the current method lies in its ability to bypass traditional time-stepping procedures. Instead, it directly solves the linear systems \eqref{eq:nnmm3Hi2}) globally using a single, coarse SGG time grid. This is particularly beneficial in diffusion-dominated scenarios, where accurate temporal resolution with exponential convergence can be achieved even with very coarse grids. By circumventing time-stepping, the method minimizes temporal error accumulation, a common pitfall in such methods, and ensures a more faithful representation of the diffusion process.

\begin{figure}[ht!]
\begin{center}
\resizebox{0.5\textwidth}{!}{%
\begin{tikzpicture}[node distance=1.5cm]

  \tikzstyle{startstop} = [rectangle, rounded corners, minimum width=2.5cm, minimum height=0.8cm,text centered, draw=black, fill=red!30]
  \tikzstyle{process} = [rectangle, minimum width=2.5cm, minimum height=0.8cm, text centered, draw=black, fill=orange!30]
  \tikzstyle{decision} = [rectangle, minimum width=2.5cm, minimum height=0.8cm, text centered, draw=black, fill=yellow!30]
  \tikzstyle{arrow} = [thick,->,>=stealth]

  \node (start) [startstop] {The FGIG Algorithm};

  \node (step1) [process, below of=start] {Determine the DFT coefficients of $u_0$ using \eqref{eq:dfjvdhkb1}};
  
  \node (step2) [process, below of=step1] {Construct the collocation matrix using \eqref{eq:habiby1}};
  
  \node (step3) [process, below of=step2] {Solve the linear systems \eqref{eq:nnmm3Hi2_2h1} for the first N/2 coefficients};
  
  \node (step4) [process, below of=step3] {Determine the last N/2 coefficients using \eqref{eq:Remark1}};
  
  \node (step5) [process, below of=step4] {Determine the zeroth-coefficient using \eqref{eq:Smart1HH1}};
  
  \node (step6) [process, below of=step5] {Determine the collocated solution using \eqref{eq:FI1nn1Dec222}};
  
  \node (decision1) [decision, below of=step6] {Collocated SSD required?};
  
  \node (step7a) [process, below of=decision1] {Determine the collocated SSD using \eqref{eq:TSSKR1}};
  
  \node (decision2) [decision, below of=step7a] {Solution/SSD at noncollocation points required?};
  
  \node (step8a) [process, below of=decision2] {Determine the necessary coefficients' values using \eqref{sec:ort:eq:Lagint1}};
  
  \node (step8b) [process, below of=step8a] {Compute the required solution's values using  \eqref{eq:AS1} and \eqref{eq:FI1nn1}};
  
  \node (end) [startstop, below of=step8b] {End};

  \draw [arrow] (start) -- (step1);
  \draw [arrow] (step1) -- (step2);
  \draw [arrow] (step2) -- (step3);
  \draw [arrow] (step3) -- (step4);
  \draw [arrow] (step4) -- (step5);
  \draw [arrow] (step5) -- (step6);
  \draw [arrow] (step6) -- (decision1);
  \draw [arrow] (decision1) -- node[anchor=east] {yes} (step7a);
  \draw [arrow] (step7a) -- (decision2);
  \draw [arrow] (decision2) -- node[anchor=east] {yes} (step8a);
  \draw [arrow] (step8a) -- (step8b);
  \draw [arrow] (step8b) -- (end);
  \draw [arrow] (decision1.west) -- ++(-3.5,0) node[anchor=south] {no} |- (end);
  \draw [arrow] (decision2.west) -- ++(-1.5,0) node[anchor=south] {no} |- (end);
  
\end{tikzpicture}
}
\end{center}
\caption{\centering Flowchart of the FGIG Algorithm.}
\label{fig:fgig_flowchart}
\end{figure}
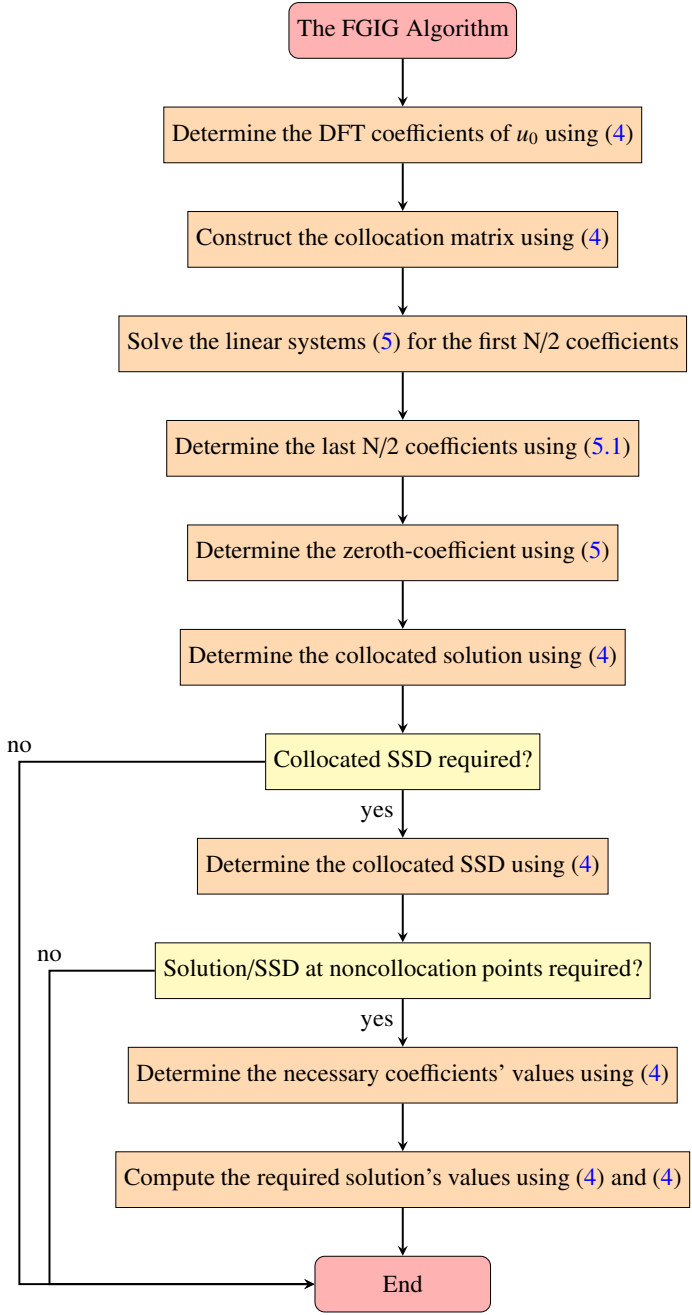

\section{Computational Cost and Time Complexity}
\label{sec:CCC1}
In Section \ref{sec:FGIG1}, we demonstrated that only $N$ Fourier coefficients of the series \eqref{eq:FI1nn1} need to be determined, as the remaining coefficient can be computed directly using Eq. \eqref{eq:NKR1}. Remarkably, the following theorem demonstrates that only the first $N/2$ coefficients need to be determined, since the remaining coefficients can be efficiently computed using a conjugate symmetry condition.

\begin{thm}\label{thm:Remark1}
The time-dependent Fourier coefficients, $\tilpsi_{-N/2:N/2}$, satisfy the conjugate symmetry condition 
\begin{equation}\label{eq:Remark1}
\tilpsi_{-n} = \tilpsi_{n}^*,\quad \forall n \in \MBKcanc{N}.
\end{equation}
\end{thm}
\begin{proof}
Let $k = n \in \MBKcanc{N}$, and notice first that 
\begin{equation}
\alpha_{-n} = \omega_{-n} (\nu \omega_{-n} + \mu i) = -\omega_n (-\nu \omega_n + \mu i) = \alpha_n^*.
\end{equation}
Taking the conjugate of the $n$th System \eqref{eq:nnmm3Hi2_2} while realizing the conjugate symmetries of $\hu_{0,n}$ and $\alpha_n$ yields
\begin{gather}\label{eq:nnmm3Hi2_2JJ1}
\left({}_T\F{A}^{(n,M)}\;{\tilpsi}_{n,0:M}\right)^* = \hu_{0,n}^* \bmone_{M+1}\\
\Rightarrow \left(\F{I}_{M+1} + \alpha_n\;{}_{T}\F{Q}\right)^*\;{\tilpsi}_{n,0:M}^* = \hu_{0,-n} \bmone_{M+1}\\
\Rightarrow \left(\F{I}_{M+1} + \alpha_{-n}\;{}_{T}\F{Q}\right)\;{\tilpsi}_{n,0:M}^* = \hu_{0,-n} \bmone_{M+1},
\end{gather}
since $\F{I}_{M+1}$ and ${}_{T}\F{Q}$ are real matrices. This matches the system for 
$k = -n$, from which the proof is established. 
\end{proof}

The aftermath of Theorem \ref{thm:Remark1} is at least twofold: (i) It shows that we only need to solve the following $N/2$ systems of linear equations:
\begin{equation}\label{eq:nnmm3Hi2_2h1}
{}_T\F{A}^{(n,M)}\;{\tilpsi}_{n,0:M} = \hu_{0,n} \bmone_{M+1},\quad \forall n \in \MBN_{N/2},
\end{equation}
for the first $N/2$ Fourier coefficients, instead of solving the $N$ systems \eqref{eq:nnmm3Hi2_2}; the last $N/2$ coefficients are computed efficiently using Eq. \eqref{eq:Remark1}, and (ii) it provides further a more efficient alternative for computing $\tilpsi_0$ using the following formula:
\begin{equation}\label{eq:Smart1HH1}
\tilpsi_0(t) = -2 \sum\limits_{k \in \MBN_{N/2}} {\Re\left(\tilpsi_k(t)\right)},
\end{equation}
which requires adding the real parts of only half the series terms of Eq. \eqref{eq:Smart1}.

Since the first-order barycentric GIM, $\F{Q}$, is typically a dense, non-symmetric matrix, solving each system in \eqref{eq:nnmm3Hi2_2h1} using highly optimized direct algorithms such as LU decomposition has a computational complexity of approximately $O(M^3)\,\foralll M$. Consequently, the total computational cost of solving Systems \eqref{eq:nnmm3Hi2_2h1} is $O(N M^3)$ $\foralll M, N$. The computation of ${}_{N,M}u$ via \eqref{eq:FI1nn1Dec222} requires $O(M N^2)$ operations for the matrix multiplication, plus $O(M N)$ scalar multiplications required by the Kronecker product. Thus, the total cost, including the cost of adding the resulting matrices, is $O(M N^2)\,\foralll M, N$. This analysis shows that the overall cost of solving the linear systems and recovering the collocated approximate solution, ${}_{N,M}u$, is 
\[\left\{ \begin{array}{l}
O(N M^3),\quad \foralll M/N,\\
O(M N^2),\quad \foralll N/M.
\end{array} \right.\]

If the collocated spatial derivative approximation, ${}_{N,M}u_x$, is also required, one may estimate the additional cost as follows: First, observe that computing the Kronecker product in Eq. \eqref{eq:FI1nn1Dec2223} has a cost of $O(N^2)$. Similarly, the Hadamard product also incurs a cost of $O(N^2)$. The subsequent matrix multiplication between the resulting matrix and the reshaped matrix costs $O(M N^2)$. Finally, scaling each entry of the resulting matrix by $i$ incurs a cost of $O(N M)$, which is negligible compared to the matrix multiplication cost. Therefore, the total additional cost is $O(M N^2)$, which matches the computational cost order of recovering ${}_{N,M}u$.

Parallel computing can significantly reduce the overall wall-clock time required to solve the linear systems. Since the $N/2$ systems are independent of each other, they can be distributed across multiple cores or processors, allowing simultaneous computation. In particular, assuming that there are $P$ available processing cores, and the $N/2$ systems are distributed evenly across these cores, the total computational time using parallel computing, $\ME{T}_{\text{parallel}}$, is approximately:
\begin{equation}\label{eq:Wow1}
\ME{T}_{\text{parallel}} = \frac{\ME{T}_{\text{serial}}}{\min\{P,N/2\}} + \ME{T}_{\text{overhead}},
\end{equation}
where $\ME{T}_{\text{serial}} = O(N M^3)$ is the serial time to solve the $N/2$ systems, and $\ME{T}_{\text{overhead}}$ is the additional time incurred due to the parallel computing process itself from tasks such as communication, synchronization, and workload distribution. These are generally negligible if the workload per system, $O(M^3)$, is sufficiently large compared to the overhead. If the overheads are minimal, the computational time is reduced significantly to approximately $O(N M^3/\min\{P,N/2\})$ by Eq. \eqref{eq:Wow1}. It is important to understand here that parallel computing is not  useful for small datasets, as the associated overhead can outweigh the potential gains from parallelization. In such cases, a standard sequential ``for'' loop might be more efficient. It is noteworthy to mention also here that the constant nature of the matrix $\F{Q}$ allows us to precompute and store it prior to executing the code for solving these systems, reducing the runtime overhead. Moreover, the numerical quadrature induced by $\F{Q}$ (or ${}_T\F{Q}$) through matrix-vector multiplication converges exponentially, enabling nearly exact approximations with relatively few quadrature nodes \cite{Elgindy201382,Elgindy20161,Elgindy20171,elgindy2018high}. Consequently, Systems \eqref{eq:Smart1HH1} can be solved efficiently and with high accuracy using a relatively small $M$ values, significantly improving the computational efficiency without compromising precision.

\section{Error and Stability Analyses}
\label{sec:ESA1}
Let $E_N = u - {}_Nu = \C{I}_x^{(x)} \psi - {}_N\C{I}_x^{(x)} \psi\,\forall N \in \MBZeP$ denote the truncation error of the truncated Fourier series expansion \eqref{eq:FI1nn1}. The squared $L^2$-norm of the error is defined as:
\begin{equation}
\left\|E_N\right\|^2 = \C{I}_T^{(t)} {\C{I}_L^{(x)} {|E_N|^2}} = L\,\C{I}_T^{(t)} {\sum\limits_{|k| > N/2} {|\tilpsi_k(t)|^2}},
\end{equation}
cf. \cite[Proof of Theorem 4.2]{elgindy2024numerical}. Thus, the error depends on the sum of the squared magnitudes of the omitted Fourier coefficients. For analytic solutions, Fourier coefficients decay exponentially with $|k|$ as shown by the following theorem.

\begin{thm}[{\cite[Theorems 4.1 and 4.2]{elgindy2024numerical}}]\label{thm:0}
Let $t \in \FOmega_T$, and suppose that $\C{I}_x^{(x)} \psi \in {{\C{A}_{L,\beta}}}$ $\foralls \beta > 0$, is approximated by the $N/2$-degree, $L$-sp, truncated Fourier series \eqref{eq:FI1nn1}, then 
\begin{subequations}
\begin{gather}
|\tilpsi_k| = O\left({e^{ - {\omega _{\left|k\right| \beta }}}}\right),\quad \text{ as }\left|k\right| \to \infty,\label{eq:fhatkMAR312021new1}\\
{\left\| {\C{I}_x^{(x)} \psi - {}_N\C{I}_x^{(x)} \psi} \right\|} = O\left(e^{ -{\omega _{N \beta/2}}}\right),\quad \text{ as }N \to \infty.\label{eq:Thm1}
\end{gather}
Moreover, if $\C{I}_x^{(x)} \psi$ is $\beta$-analytic, then
\begin{equation}\label{eq:Thm2}
{\left\| {\C{I}_x^{(x)} \psi - {}_N\C{I}_x^{(x)} \psi} \right\|} = 0,\quad \forall N \in \MBZeP.
\end{equation}
\end{subequations}
\end{thm}
The rapid decay of Fourier coefficients for sufficiently smooth functions is a hallmark of diffusion-dominated problems, where the solution tends to smooth out over time. The FGIG method effectively exploits this property by employing a Fourier basis, enabling an accurate representation of the solution with relatively few terms. This contributes to the method's efficiency and accuracy in diffusion-dominated regimes.

For nonsmooth solutions, the coefficients decay at a polynomial rate depending on the degree of smoothness, as shown by the following theorem.

\begin{thm}[{\cite[Theorem A.1 and A.2]{Elgindy2023a}}]\label{thm:00}
Let $t \in \FOmega_T$, and suppose that $\C{I}_x^{(x)} \psi \in {{\C{H}_L^s}} \foralls s \in \MBZzerP$ is approximated by the $N/2$-degree, $L$-sp truncated Fourier series \eqref{eq:FI1nn1}, then
\begin{gather}
|\tilpsi_k| = O\left( {{{\left| {{k}} \right|}^{ - s - 1}}} \right),\quad \text{ as }\left|k\right| \to \infty,\label{eq:fhatkMAR312021new10}\\
{\left\| {\C{I}_x^{(x)} \psi - {}_N\C{I}_x^{(x)} \psi} \right\|} = O\left(N^{-s-1/2}\right),\quad \text{ as }N \to \infty.\label{eq:Thm11}
\end{gather}
\end{thm}

Discontinuities or shocks features in $u$ affect the coefficients and result in the well renowned Gibbs phenomenon, where the truncated Fourier series near a discontinuity exhibits an overshoot or undershoot that does not diminish as the number of terms in the sum increases.

A crucial step in the implementation of the FGIG method lies in the solution of the linear system \eqref{eq:nnmm3Hi2_2h1}. To analyze the sources of error in solving the given linear system, we must consider errors arising from numerical stability, discretization, and condition numbers, as well as how these depend on $N, M, \mu$, and $\nu$. The right-hand-side of the linear system presents the $n$th Fourier interpolation coefficient of $u_0$, which reflects how $u_0$ is approximated by a finite number of frequencies $N$. The convergence rates of the interpolation error associated with Eq. \eqref{eq:FI1newu01} are analyzed in the following two corollaries, specifically focusing on their dependence on the smoothness of the underlying function space.

\begin{cor}[{\cite[Corollary 4.1]{elgindy2024numerical}}]\label{cor:1}
Suppose that $u_0 \in {{\C{A}_{L,\beta}}} \foralls \beta > 0$ is approximated by the $N_0/2$-degree, $L$-periodic Fourier interpolant \eqref{eq:FI1newu01}, $\foralla N_0 \in \MBZeP$, then 
\begin{subequations}
\begin{equation}\label{eq:FTS1AE1hi1}
\left\|u_0 - {I_{N_0}}u_0\right\| = O\left(e^{-\omega_{N_0 \beta/2}}\right),\quad \text{as }N_0 \to \infty.
\end{equation}
Moreover, if $u_0$ is $\beta$-analytic, then 
\begin{equation}\label{eq:Thm2hi2}
\left\|u_0 - {I_{N_0}}u_0\right\| = 0,\quad \forall N_0 \in \MBZeP.
\end{equation}
\end{subequations}
\end{cor}

\begin{cor}[{\cite[Corollary A.1]{Elgindy2023a}}]\label{cor:11}
Suppose that $u_0 \in {{\C{H}_L^s}} \foralls s \in \MBZ^+$ is approximated by the $N_0/2$-degree, $L$-periodic Fourier interpolant \eqref{eq:FI1newu01}, $\foralla N_0 \in \MBZeP$, then 
\begin{equation}\label{eq:FTS1AE1hi126Apr2021}
\left\|u_0 - {{I_{N_0}u_0}} \right\| = O\left( {{N_0^{-s-1/2}}} \right),\quad \text{as }N_0 \to \infty.
\end{equation}
\end{cor}

Now, we turn our attention to the left-hand-side of the linear system \eqref{eq:nnmm3Hi2_2h1} whose coeficient matrix mainly consists of the SGIM that is scaled by the factor $\alpha_n$. The error of the SGG quadrature induced by the SGIM, ${}_T\F{Q}$, for sufficiently smooth functions, can be described in closed form by the following theorem.

\begin{thm}[{\cite[Theorem 4.1]{Elgindy20161}}]\label{subsec:err:thm1}
Let $f \in C^{M + 1}(\FOmega_T)$, be interpolated by the shifted Gegenbauer (SG) polynomials at the SGG nodes, $\hatt_{M,0:M}^{(\lambda)} \in \MBSG_{T,M}^{\lambda}$. Then there exist a matrix ${}_T\F{Q} = ({{}_Tq_{l,j}}),\,0 \le l,j \le M$, and some numbers $\xi_l = \xi\left(\hatt_{M,l}^{(\lambda)}\right) \in \IFOmega_T\,\forall l \in \MBJP_M$, satisfying
\begin{equation}\label{subsec:err:eq:squadki1}
\C{I}_{\hatt_{M,l}^{(\lambda )}}^{(t)} f = {}_T\F{Q}_l\,f_{0:M} + E_{T,M}^{(\lambda )}\left( {\hatt_{M,l}^{(\lambda )},{\xi _l}} \right),
\end{equation}
where $f_{0:M} = f\left(\bmt_M^{\lambda}\right)$,
\begin{equation}\label{sec1:eq:errorkimohat}
E_{T,M}^{(\lambda )}\left( {\hatt_{M,l}^{(\lambda )},{\xi _l}} \right) = \frac{{{f^{(M + 1)}}({\xi _l})}}{{(M + 1)!\,K_{T,M + 1}^{(\lambda )}}} \C{I}_{\hatt_{M,l}^{(\lambda )}}^{(t)} {G_{T,M + 1}^{(\lambda)}},
\end{equation}
\begin{equation}
K_{T,j}^{(\lambda)} = \frac{2^{2j-1}}{T^j} \frac{\Gamma(2 \lambda+1) \Gamma(j+\lambda)}{\Gamma(\lambda+1) \Gamma(j+2\lambda)}\quad \forall j \in \MBZzerP,
\end{equation}
is the leading coefficient of the $j$th-degree SG polynomial, $G_{T,j}^{(\lambda)}(t)$. 
\end{thm}

The following theorem shows that the SG quadrature formula converges exponentially fast for sufficiently smooth functions. Its proof can be immediately derived from \cite[Theorem A.5]{Elgindy2023a} in the absence of domain partitioning.

\begin{thm}[{\cite[Theorem A.5]{Elgindy2023a}}]\label{thm:Jan212022}
Let ${\left\| {{f^{({M} + 1)}}} \right\|_{L^{\infty}(\FOmega_T)}} = A \in \MBRzerP$, where the constant $A$ is independent of $M$. Suppose also that the assumptions of Theorem \ref{subsec:err:thm1} hold true. Then there exist some constants ${D^{\left( \lambda \right)}} > 0, B_1^{\left( \lambda \right)} = {A}{D^{\left( \lambda \right)}}$, and $B_2^{\left( \lambda \right)} > 1$, which depend on $\lambda$ but are independent of $M$, such that the SG quadrature truncation error, $E_{T,{M}}^{\left( \lambda \right)}\left( {\hatt_{M,l}^{(\lambda )},\xi_l} \right)$, is bounded by
\begin{equation}
\scalebox{0.875}{$\begin{array}{l}
	\left\| {E_{{M}}^{\left( \lambda \right)}\left( {\hatt_{M,l}^{(\lambda )},\xi_l} \right)} \right\|_{L^{\infty}(\FOmega_T)} =  B_1^{\left( \lambda \right)}\,{2^{ - 2{M} - 1}}{{{e}}^{{M}}}{M}^{\lambda - {M} - \frac{3}{2}} {T^{{M} + 1}} {\hatt_{M,l}^{(\lambda )}} \times \\
	\left( {\left\{ \begin{array}{l}
	1,\quad {M} \ge 0 \wedge \lambda \ge 0,\\
	\displaystyle{\frac{{\Gamma \left( {\frac{{{M}}}{2} + 1} \right)\Gamma \left( {\lambda + \frac{1}{2}} \right)}}{{\sqrt \pi\,\Gamma \left( {\frac{{{M}}}{2} + \lambda + 1} \right)}}},\quad M \in \MBZOP \wedge  - \frac{1}{2} < \lambda < 0,\\
	\displaystyle{\frac{{2\Gamma \left( {\frac{{{M} + 3}}{2}} \right)\Gamma \left( {\lambda + \frac{1}{2}} \right)}}{{\sqrt \pi  \sqrt {\left( {{M} + 1} \right)\left( {{M} + 2\lambda + 1} \right)}\,\Gamma \left( {\frac{{{M} + 1}}{2} + \lambda} \right)}}},\quad M \in \MBZzereP \wedge  - \frac{1}{2} < \lambda < 0,\\
	B_2^{\left( \lambda \right)} {\left( {{M} + 1} \right)^{ - \lambda}},\quad {M} \to \infty  \wedge  - \frac{1}{2} < \lambda < 0
	\end{array} \right.} \right),
	\end{array}$}
\end{equation}
$\forall l \in \MBJ_{M}^+$.
\end{thm}

Besides the exponential convergence of SG quadrature, \citet{Elgindy2019b} proved earlier that the upper bounds of the rounding errors in the calculation of the elements of $q$th-order GIMs of size $(n + 1) \times (m + 1)$ required to construct the Gegenbauer quadrature rules are roughly of $O(m u_R)$, where $m, n$, and $q \in \MBZP$. Formula \eqref{eq:NewSGIM1} shows immediately that the rounding errors in ${}_TQ$ are of $O(T M u_R/2) = O(M u_R)$,
which is the same as the rounding errors in $\F{Q}$ up to a constant factor of $T/2$. The scaling by $T/2$ does not change the asymptotic order of the rounding error, though it could affect the constant factors involved. The scalar multiplication by $T/2$ does not affect the condition number either because it scales both the norm of $\F{Q}$ and and the norm of $\F{Q}^{-1}$ by a factor and its reciprocal:
\begin{equation}
\Kappa({}_T\F{Q}) = \left|\frac{T}{2}\right| \left\|\F{Q}\right\| \cdot \left|\frac{2}{T}\right| \left\|\F{Q}^{-1}\right\| = \Kappa(\F{Q}).
\end{equation}
Thus, the condition numbers of $\F{Q}$ and ${}_T\F{Q}$ are identical.

Now, let's analyze the error due to rounding in the collocation matrix ${}_T\F{A}^{(n,M)}$ in \eqref{eq:nnmm3Hi2_2h1}. Assume that ${}_T\F{Q}$ is perturbed by $\Delta {}_T\F{Q}$ due to rounding errors. Furthermore, we assume that $\alpha_n$ also has rounding errors, which we will denote by $\Delta \alpha_n: \left\|\Delta \alpha_n\right\| = O(u_R)$. Then the rounding errors in ${}_T\F{A}^{(n,M)}$ can be estimated by $\Delta \alpha_n {}_T\F{Q} + \alpha_n \Delta {}_T\F{Q} + \Delta \alpha_n \Delta {}_T\F{Q}$.
These error terms contribute asymptotic rounding errors in each matrix element of orders $O(u_R), O(M \alpha_n u_R)$, and $O(M u_R^2)$, respectively. We can thus express the total asymptotic error in each matrix element as $O\left(\max\left\{u_R, M \alpha_n u_R, M u_R^2\right\}\right)$. Consequently, the asymptotic error in each element of the product ${}_T\F{A}^{(n,M)} {\tilpsi}_{n,0:M}$ of the linear system \eqref{eq:nnmm3Hi2_2h1} can be directly estimated by 
\[O\left(\max\left\{u_R, M \alpha_n u_R, M u_R^2\right\} \left\|{\tilpsi}_{n,0:M}\right\|_1\right),\quad \forall n \in \MBN_{N/2}.\]
Notice how the size of Fourier coefficients is critical in determining the error. The rate at which these coefficients decay as $|n|$ increases is influenced by the smoothness of the solution function $u$, as shown earlier in this section. In particular, for smooth functions, which is often the case with a relatively low P\'{e}clet number $Pe = \mu L/\nu$, Fourier coefficients decay exponentially as $|n|$ increases, and their $l_1$-norm will be relatively small, leading to a smaller overall error, even as $M$ or $N$ increases. For nonsmooth solutions, however, especially those with discontinuities or sharp gradients, which often occur at relatively high $Pe$ values, Fourier coefficients decay much more slowly. This slower decay means that the Fourier coefficients for higher frequencies remain relatively large, and their contribution to the $l_1$-norm will be large. As a result, the error in the product can increase significantly. This shows that nonsmoothness often leads to slower convergence and higher numerical errors as expected.

The overall error in the solution will depend not only on the error in the matrix-vector product, ${}_T\F{A}^{(n,M)} {\tilpsi}_{n,0:M}$, or the relatively small errors in the discrete Fourier coefficients of $u_0$, but also on the conditioning of the collocation matrix ${}_T\F{A}^{(n,M)}$, which is given by
\begin{equation}
\text{\Kappa}\left({}_T\F{A}^{(n,M)}\right) = \left\|\F{I}_{M+1} + \alpha_n\,{}_T\F{Q}\right\| \left\|\left(\F{I}_{M+1} + \alpha_n\,{}_T\F{Q}\right)^{-1}\right\|. 
\end{equation}
Therefore, the behavior of the condition number depends on the relative size of $\alpha_n$ and the spectrum of $\alpha_n\,{}_T\F{Q}$. If $\lambdabar_{0:M}$ are the eigenvalues of ${}_T\F{Q}$, then the eigenvalues of ${}_T\F{A}^{(n,M)}, \hlambdabar_{0:M}$, are shifted-scaled versions of $\lambdabar_{0:M}$ by the scaling factor $\alpha_n: \hlambdabar_{0:M} = 1+\alpha_n \lambdabar_{0:M}$. To partially analyze the effect of this shift-scaling operation on the conditioning of the collocation matrix, let $\lambdabar_{\min}$ and $\hlambdabar_{\min}$ denote the smallest eigenvalues of ${}_T\F{Q}$ and ${}_T\F{A}^{(n,M)}$, respectively. Let also $\sigma_{\max}, \sigma_{\min}, \hsigma_{\max}$, and $\hsigma_{\min}$ denote the extreme singular values of ${}_T\F{Q}$ and ${}_T\F{A}^{(n,M)}$ in respective order. The first rows of Figures \ref{fig:0_1}-\ref{fig:0_3} display their distributions together with the distributions of $\lambdabar_{0:M}$ and $\hlambdabar_{0:M}$ for increasing values of $M$ and $\lambda$. Observe how $\lambdabar_{0:M}$ cluster gradually around $0$ as the size of the matrix increases, for all values of $\lambda$. Moreover, increasing $\lambda$ values, while holding $M$ fixed, tends to gradually spread out the eigenvalues in the complex plane away from the origin. This shows that $\lambdabar_{\min} \to 0$, as $\lambda \to -0.5$. While a near-zero eigenvalue does not directly determine the exact value of the smallest singular value, it strongly confirms the existence of a relatively small $\sigma_{\min}$, leading to a large condition number and potential numerical instability. In fact, Theorem \ref{thm:Woowmmm1} proves the existence of at least one near zero singular value of $\F{Q}$, not necessarily the smallest singular value, if $\lambda \to -0.5$. Figure \ref{fig:0_5} confirms this fact, where we can clearly see the rapid decay of the smallest singular values of $\F{Q}$ as $\lambda \to -0.5$. This analysis shows that \Kappa$(\F{Q}) = $\Kappa$\left({}_T\F{Q}\right) \to \infty$, as $\lambda \to -0.5$, with a faster growth rate as $M$ increases.  The first row of Figure \ref{fig:0_4} further supports this observation, showing a significant shift in the order of magnitude of \Kappa$\left({}_T\F{Q}\right)$ for all $M$ values as $\lambda$ gradually approaches $-0.5$. The figure also reveals that the curve of \Kappa$\left({}_T\F{Q}\right)$ initially exhibits a near $L$-shaped pattern for small $M$ values. It rapidly increases as $\lambda$ decreases below $0$ but grows slowly as $\lambda$ increases beyond $1$. This latter growth rate increases gradually for larger $M$, and the curve of \Kappa$\left({}_T\F{Q}\right)$ gradually transitions into a $U$-shaped pattern with a base in the range $0 \le \lambda \le 1$. It is striking here to observe  that the poor conditioning for $\lambda \to -0.5$ can be largely restored with a proper scaling of ${}_T\F{Q}$ followed by a shift by the identity matrix. This operation effectively shifts $\lambdabar_{\min}$ from near $0$ to near $1$, for relatively small scaling factors. In particular, assuming a relatively small $\alpha_n$, which occurs when $\mu, \nu$, and $n$ are small, the scaled eigenvalue $\alpha_n \lambdabar_j$ remains relatively small $\forall j \in \MBJP_M$, especially for large $M$, forcing the eigenvalues of ${}_T\F{A}^{(n,M)}$ to cluster around $1$. This can significantly decrease the ratio $\hsigma_{\max}/\hsigma_{\min}$. Furthermore, if $\alpha_n$ is too small, then ${}_T\F{A}^{(n,M)} \approx \F{I}_{M+1}$. Hence, \Kappa$\left({}_T\F{A}^{(n,M)}\right) \approx 1$, and the matrix is a near perfectly well-conditioned matrix. For the sake of illustration, consider the dataset $\{L = 2, T = 0.2, \mu = 0, \nu = 1, N = 4\}$. At the fundamental frequency, we can readily verify that $\hlambdabar_{0:M} = \bmone_{M+1}+\pi^2 \lambdabar_{0:M} \approx 1$, for sufficiently large $M$, since $\lambdabar_{0:M}$ cluster around the origin as evident by Figures \ref{fig:0_2} and \ref{fig:0_3}. Notice however that this result remains valid for small $M$ values only if $|\lambda|$ remains relatively small. The second rows of Figures \ref{fig:0_1}-\ref{fig:0_3} manifest the distributions of $\hlambdabar_{0:M}$, where we observe their clustering around unity, for increasing $M$ values. The second row of Figure \ref{fig:0_4} displays the astonishing decay of the condition number for all $\lambda$ and $M$ values, dropping by four orders of magnitude in the range $40 \le M \le 80$. Notice here that the conditioning relatively deteriorates as $\lambda$ continues to grow beyond nearly 1.5, and the degeneration becomes relatively clear for large $M$ values.

The above analysis assumes a small $\alpha_n$ for all cases. When $\alpha_n$ becomes relatively large, \Kappa$\left({}_T\F{A}^{(n,M)}\right)$ will depend on the interplay between the parameters $n, \mu$, and $\nu$. Assuming $\mu$ and $\nu$ are held fixed, the largest condition number in this case occurs at the largest value of $\alpha_n$, which occurs at $n = N/2$; see Theorem \ref{thm:poyopi1}. At this stage, increasing either $\mu$ or $\nu$ makes the eigenvalues more spread out in the complex plane, which will gradually increase \Kappa$\left({}_T\F{A}^{(n,M)}\right)$ until a terminal value where it ceases to increase anymore, assuming $N, \lambda < \infty$. This is because ${}_T\F{A}^{(n,M)} \sim \alpha_n {}_T\F{Q}\,\foralll \alpha_n$, thus, \Kappa$\left({}_T\F{A}^{(n,M)}\right) \to\,$\Kappa${}_T\F{Q}$. This demonstrates that for highly oscillatory solutions, where numerous Fourier terms are necessary to accurately represent the high-frequency oscillations, the condition number of the resulting linear system will be, at worst, comparable to that of the integration matrix itself. In overall, for larger $n$, the effect of both $\mu$ and $\nu$ becomes more pronounced, and the balance between these two constants is crucial in the sense that reducing their values can improve the conditioning of the linear system. Figures \ref{fig:0_6}-\ref{fig:0_9} demonstrate the effect of $\mu$ and $\nu$ on the conditioning as their values grow large.


\begin{figure}[t]
\centering
\includegraphics[scale=0.65]{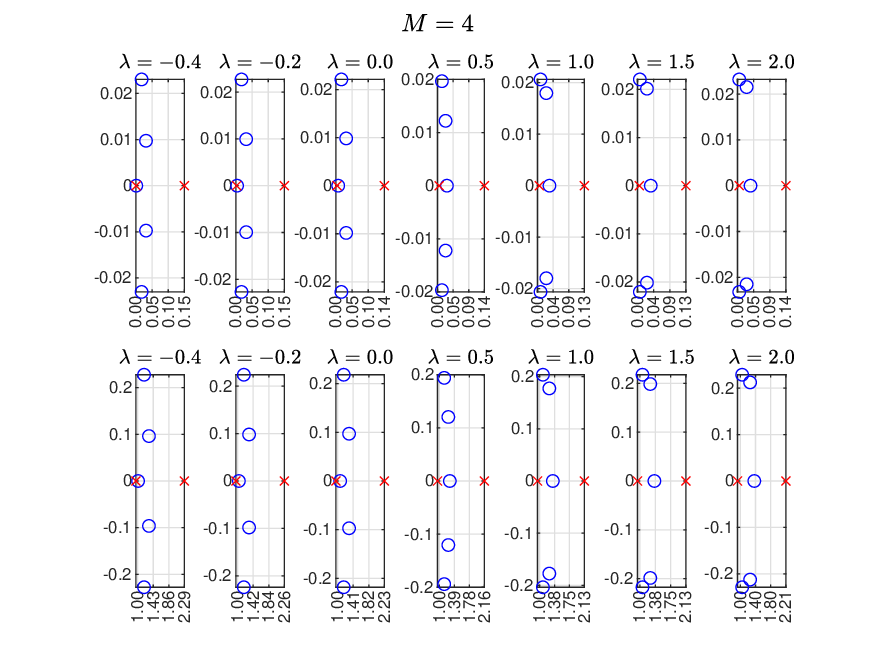}
\caption{The distribution of the eigenvalues (blue circles) and the extreme singular values (red x marks) of ${}_T\F{Q}$ (first row) and ${}_T\F{A}^{(1,M)}$ (second row) for $M = 4$ and $\lambda = -0.4, -0.2, 0:0.5:2$. All plots were generated using $L = 2$, $T = 0.2, \mu = 0, \nu = 1$, and $N = 4$.}
\label{fig:0_1}
\end{figure}

\begin{figure}[t]
\centering
\includegraphics[scale=0.65]{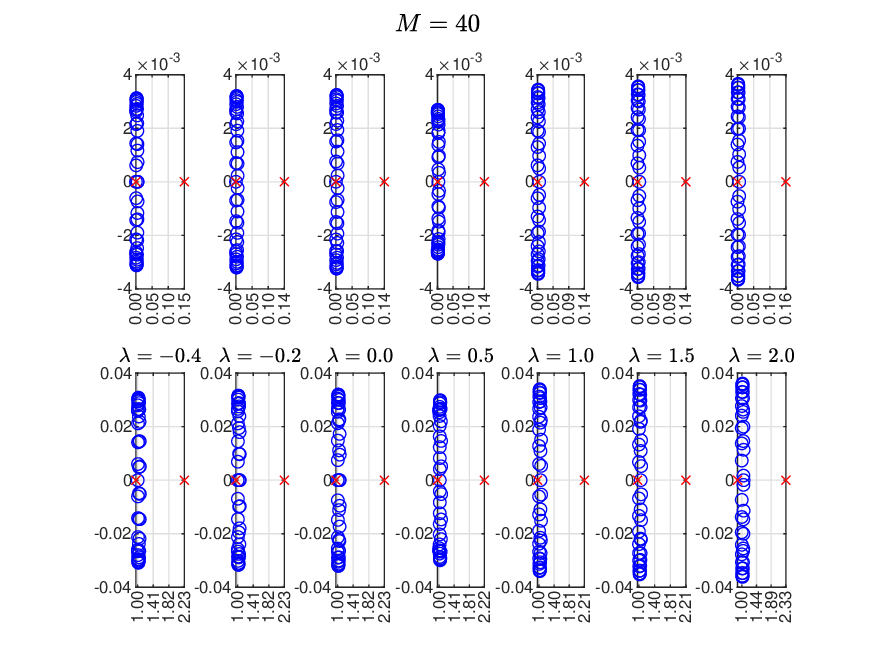}
\caption{The distribution of the eigenvalues (blue circles) and the extreme singular values (red x marks) of ${}_T\F{Q}$ (first row) and ${}_T\F{A}^{(1,M)}$ (second row) for $M = 40$ and $\lambda = -0.4, -0.2, 0:0.5:2$. All plots were generated using $L = 2$, $T = 0.2, \mu = 0, \nu = 1$, and $N = 4$.}
\label{fig:0_2}
\end{figure}

\begin{figure}[t]
\centering
\includegraphics[scale=0.65]{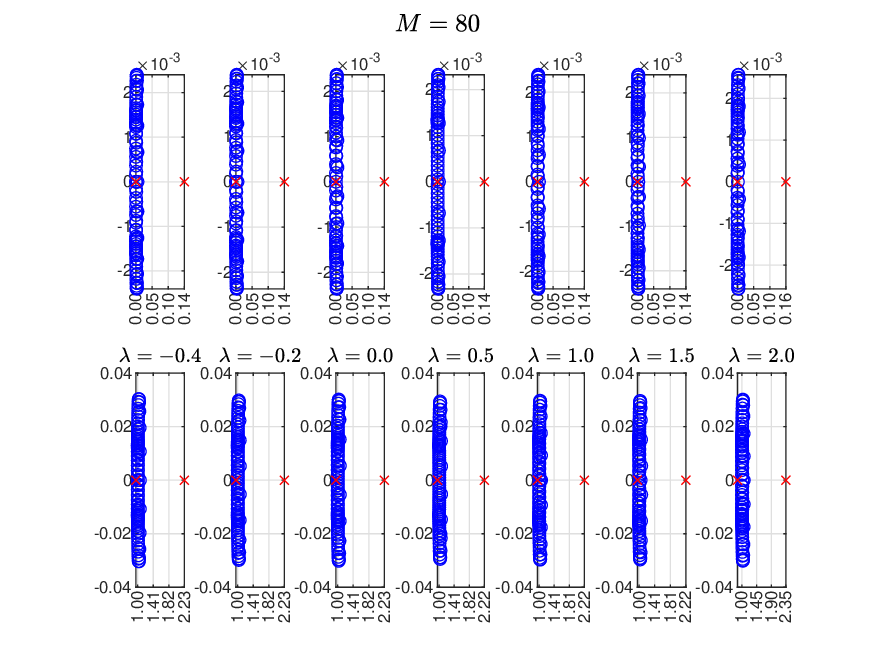}
\caption{The distribution of the eigenvalues (blue circles) and the extreme singular values (red x marks) of ${}_T\F{Q}$ (first row) and ${}_T\F{A}^{(1,M)}$ (second row) for $M = 80$ and $\lambda = -0.4, -0.2, 0:0.5:2$. All plots were generated using $L = 2$, $T = 0.2, \mu = 0, \nu = 1$, and $N = 4$.}
\label{fig:0_3}
\end{figure}

\begin{figure}[t]
\centering
\includegraphics[scale=0.5]{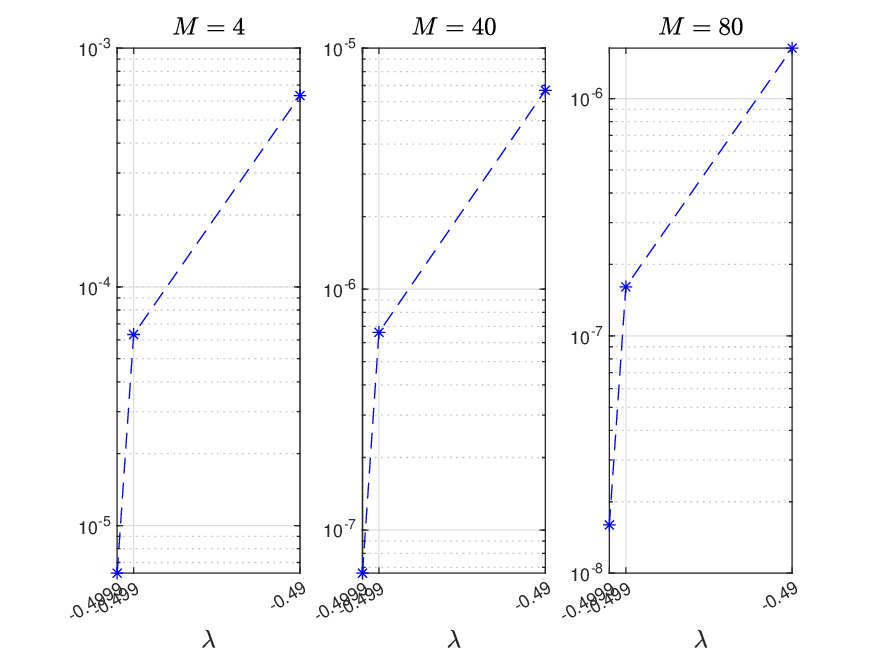}
\caption{The values of the smallest singular values of $\F{Q}$ for $\lambda = -0.49,-0.499,-0.4999$ and $M = 4, 40$, and $80$.}
\label{fig:0_5}
\end{figure}

\begin{figure}[t]
\centering
\includegraphics[scale=0.55]{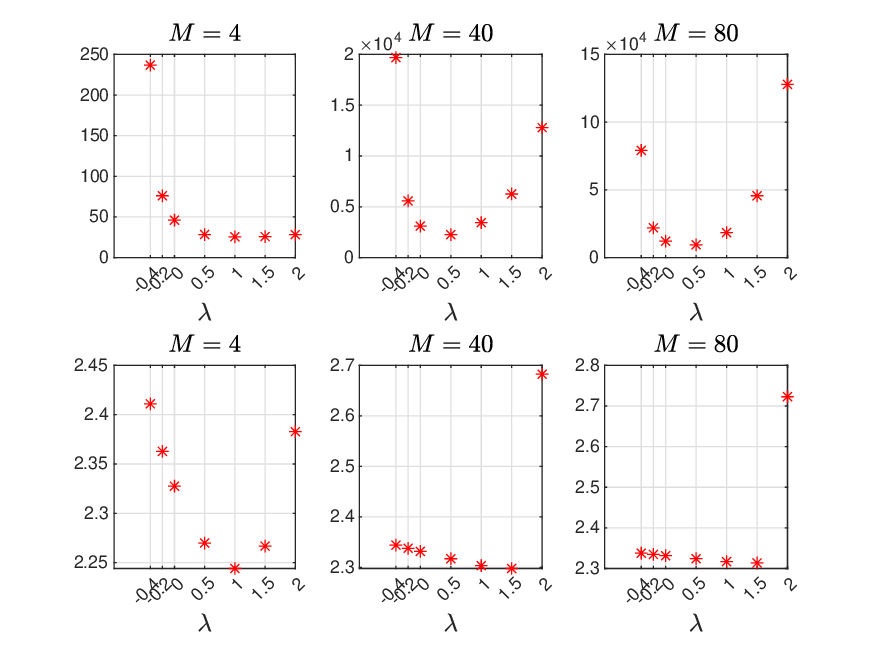}
\caption{The values of \Kappa$({}_T\F{Q})$ (first row) and \Kappa$\left({}_T\F{A}^{(1,M)}\right)$ (second row) for $M = 4, 40, 80$ and $\lambda = -0.4, -0.2, 0:0.5:2$. All plots were generated using $L = 2$, $T = 0.2, \mu = 0, \nu = 1$, and $N = 4$.}
\label{fig:0_4}
\end{figure}

\begin{figure}[t]
\centering
\includegraphics[scale=0.55]{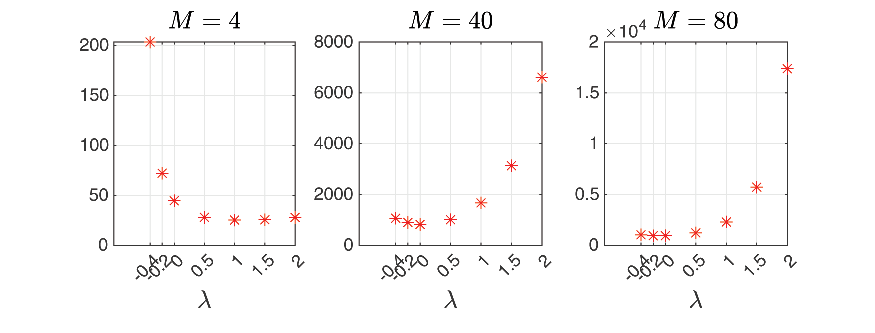}
\caption{The values of \Kappa$\left({}_T\F{A}^{(N/2,M)}\right)$ for $M = 4, 40, 80$ and $\lambda = -0.4, -0.2, 0:0.5:2$. All plots were generated using $L = 2, T = 0.2, \mu = 1, \nu = 1$, and $N = 50$.}
\label{fig:0_6}
\end{figure}

\begin{figure}[t]
\centering
\includegraphics[scale=0.55]{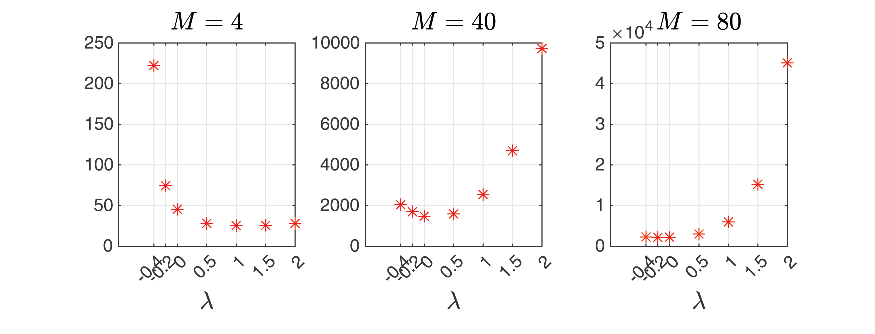}
\caption{The values of \Kappa$\left({}_T\F{A}^{(N/2,M)}\right)$ for $M = 4, 40, 80$ and $\lambda = -0.4, -0.2, 0:0.5:2$. All plots were generated using $L = 2, T = 0.2, \mu = 100, \nu = 1$, and $N = 50$.}
\label{fig:0_7}
\end{figure}

\begin{figure}[t]
\centering
\includegraphics[scale=0.55]{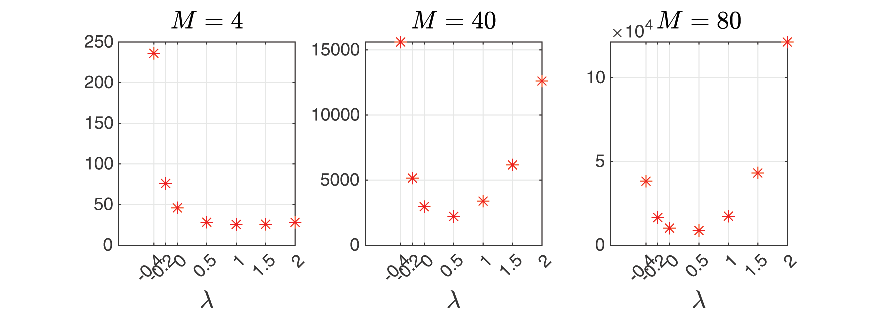}
\caption{The values of \Kappa$\left({}_T\F{A}^{(N/2,M)}\right)$ for $M = 4, 40, 80$ and $\lambda = -0.4, -0.2, 0:0.5:2$. All plots were generated using $L = 2, T = 0.2, \mu = 1, \nu = 50$, and $N = 50$.}
\label{fig:0_8}
\end{figure}

\begin{figure}[t]
\centering
\includegraphics[scale=0.5]{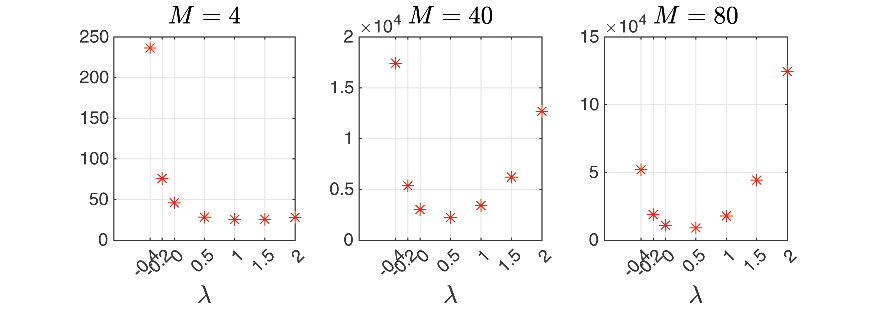}
\caption{The values of \Kappa$\left({}_T\F{A}^{(N/2,M)}\right)$ for $M = 4, 40, 80$ and $\lambda = -0.4, -0.2, 0:0.5:2$. All plots were generated using $L = 2, T = 0.2, \mu = 100, \nu = 100$, and $N = 50$.}
\label{fig:0_9}
\end{figure}

The above analysis emphasizes that the proper choice of the Gegenbauer parameter $\lambda$ can significantly influence the stability and conditioning of the SGIM, which is crucial for the stability and convergence of spectral integration methods employing Gegenbauer polynomials, in general. Besides our new findings here on the strong relationship between $\lambda$ and the conditioning of the matrix, it is important to mention also that the barycentric representation of the SGIM employed in this work generally improves the stability of Gegenbauer quadratures, since it mitigates numerical instabilities that can arise from the direct evaluation of Gegenbauer polynomials, especially for large degrees or specific $\lambda$ values \cite{Elgindy20171}. Since the stability characteristics of the SGIM are directly inherited from the GIM, a suitable rule of thumb that can be drawn from this study and verified by Figure \ref{fig:0_4} is to select $\lambda$ within the interval $(-1/2+\varepsilon, 2]$ for a relatively small $p = \max\{N,\mu,\nu\}$, where $\varepsilon$ is a relatively small and positive parameter. For $\foralll p$, Figures \ref{fig:0_6}-\ref{fig:0_9} suggest to shrink this recommended interval into $[0, 0.5]$ to a maintain a relatively low condition number. 

Besides the conditioning matter, we draw the attention of the reader to the important fact that the Gegenbauer weight function associated with the GIM diminishes rapidly near the boundaries $x = \pm 1$ for increasing $\lambda > 2$, which forces the Gegenbauer quadrature to depend more heavily on the behavior of the integrand near the central part of the interval, increasing sensitivity to errors, and forcing the quadrature to become more extrapolatory. The parameter $\varepsilon$ is a small buffer added to $-1/2$ to maintain the numerical stability of the Gegenbauer quadrature, since the Gegenbauer polynomials of increasing orders grow rapidly as $\lambda \to -1/2$. Moreover, for sufficiently smooth functions and large spectral expansion terms, the truncated expansion in the shifted Chebyshev quadrature, associated with $\lambda = 0$, is optimal for the $L^{\infty}$-norm approximations of definite integrals, while the shifted Legendre quadrature, associated with $\lambda = 0.5$, is optimal for the $L^2$-norm approximations \cite{Elgindy201382}.

\section{The Semi-Analytical FGIG (SA-FGIG) Method}
\label{sec:FGIG12}
We show in this section how to compute the time-dependent Fourier coefficients analytically, and use them to synthesis the solution $u$. To this end, notice that Eq. \eqref{eq:nnmm3} is a Volterra integral equation of the first kind for $\tilpsi_n\,\foralle n$. To obtain its solutions in closed form, we differentiate its both sides w.r.t. $t$ to get
\begin{equation}\label{eq:nnmm3mn1}
\frac{d\tilpsi_n(t)}{dt} + \alpha_n \tilpsi_n(t) = 0,\quad \forall n \in \MBN_{N/2}.
\end{equation}
This is a first-order linear ODE for $\tilpsi_n$ with the initial condition $\tilpsi_n(0) = \hu_{0,n}\,\forall n \in \MBN_{N/2}$, which can be derived easily from Eq. \eqref{eq:nnmm3}. Its solution is therefore given by
\begin{equation}\label{eq:SGV121}
\tilpsi_n(t) = \hu_{0,n} e^{-\alpha_n t},\quad \forall n \in \MBN_{N/2}.
\end{equation}
Notice here that $\Re(\alpha_n) = \nu \omega_n^2$ controls the super-exponential decay of $\tilpsi_n$ due to the quadratic dependence on $n$ in the exponent, for $\nu > 0$. Consequently, for $\nu > 0$, higher-frequency modes, corresponding to large $|n|$, decay much faster. This is characteristic of diffusion, which dissipates high-frequency components rapidly, leading to the smoothing of the solution $u$. On the other hand, the oscillatory behavior of $\tilpsi_n$ arises from two sources: (i) the discrete Fourier coefficient $\hu_{0,n}$, which is generally complex, and introduces an initial phase that defines the starting point of the oscillation for each mode, and (ii) $\Im(\alpha_n) = \mu \omega_n$, which introduces a time-dependent phase shift, $e^{-i \mu \omega_n t}$, that causes the oscillations to evolve over time. The frequency of these oscillations increases linearly with $|n|$. This reflects advection, which causes phase shifts in the modes without affecting their amplitude. The combined effects of $\nu$ and $\mu$, encapsulated by the parameter $\alpha_n$, cause the modes to experience amplitude decay that is dominated by $\nu$ and a phase shift that is determined by $\mu$. 

Using Eq. \eqref{eq:Smart1HH1}, the conjugate symmetry condition \eqref{eq:Remark1}, and the closed form expression of Fourier coefficients \eqref{eq:SGV121}, we can now write the Fourier series of the semi-analytic (SA) solution, $u^{sa}$, as follows:
\begin{align}
{}_Nu^{sa}(x,t) &=  2 \Re\left(\sum\limits_{k \in \MBN_{N/2}} {{\tilpsi_k(t)}\; {e^{i{\omega _k}x}}}\right) - 2 \sum\limits_{k \in \MBN_{N/2}} {\Re\left(\tilpsi_k(t)\right)} + g(t)\\
&= 2 \Re\left(\sum\limits_{k \in \MBN_{N/2}} {\hu_{0,k} e^{-\alpha_k t + i{\omega _k}x}}\right) - 2 \sum\limits_{k \in \MBN_{N/2}} {\Re\left(\hu_{0,k} e^{-\alpha_k t}\right)} + g(t).\label{eq:FI1nn1kl1} 
\end{align}
The SA-SSD, ${}_Nu_x^{sa}$, is therefore given by
\begin{align}\label{eq:FI1nn1kl13}
{}_Nu_x^{sa}(x,t) &= -2 \Im\left(\sum\limits_{k \in \MBN_{N/2}} {\omega_k \hu_{0,k} e^{-\alpha_k t + i{\omega _k}x}}\right). 
\end{align}

\section{Computational Results}
\label{sec:CRAC1}
This section presents a series of numerical experiments to validate the accuracy and efficiency of the proposed FGIG and SA-FGIG methods. We consider three benchmark problems with known analytical solutions, allowing for a direct comparison between the numerical and exact solutions. Each method's performance is evaluated using error norms and computational time. We measure the error in the approximate solution of Problem S at the augmented collocation points set $\MBS_N^L \times \MBSG_{T,M}^{\lambda,+}$ using the absolute error function given by
\begin{equation}\label{eq:AEF1}
E_u(x,t) = |u(x,t)-{}_Nu(x,t)|,
\end{equation}
and the discrete norm at $t = T$ given by
\begin{equation}\label{eq:EDN1}
\text{DNE}_u = \left\|u(x,T)-{}_Nu(x,T)\right\|_N = \scalebox{0.9}{$\sqrt{\frac{L}{N} \sum\limits_{j \in \MBJ_N} {\left(u(x_{N,j},T)-{}_Nu(x_{N,j},T)\right)^2}}$}.
\end{equation}
The superscript ``sa'' is added to the above two error notations when the solution is approximated by the SA-solution obtained through \eqref{eq:FI1nn1kl1}.

\textbf{Test Problem 1.} Consider Problem S with $\mu = 0, \nu = 1$, and $L = 2$ with initial and boundary functions $u_0(x) = \sin(\pi x)$ and $g(t) = 0$. The exact solution of this problem is $u(x,t) = e^{-\pi^2 t} \sin(\pi x)$ \cite{jena2023one}. Figure \ref{fig:1_1} shows the surfaces of the exact solution and its approximations obtained by the FGIG and the SA-FGIG methods as well as their corresponding errors for some parameter values. Both methods achieve near machine precision in resolving the solution and the SSD on a relatively coarse mesh grid of size $22 \times 13$. Figure \ref{fig:2_1} shows the surfaces of the logarithms of the discrete error norms for both methods at the terminal time, $t = 0.2$, for certain ranges of the parameters $N, M$, and $\lambda$. When $N$ is fixed, the linear decay of the surface for increasing $M$ values indicates the exponential convergence of the FGIM. The nearly flat surface along $N$ while holding $M$ fixed in the plot associated with the FGIM occurs because the error is dominated by the discretization error of System \eqref{eq:nnmm3Hi1}, controlled by the size of the temporal linear system. To observe the theoretical exponential decay with $N$, it is necessary to increase the time mesh points grid so that the temporal discretization error becomes negligible compared to the spatial error. This interplay between $N$ and $M$ reflects the dependence of the total error on both spatial and temporal resolution. The SA-FGIM method, on the other hand, exhibits much smaller logarithmic error norms in the range $-16.1$ to $-16.8$, indicating a relatively flat surface. This is expected because the exact solution is $\beta$-analytic in $x\,\forall L < 2$; cf. Theorem \ref{thm:MP1}. Consequently, Fourier truncation and interpolation errors vanish $\forall N \in \MBZeP$ by Theorem \ref{thm:0} and Corollary \ref{cor:1}. Furthermore, the numerical errors in double-precision arithmetic quickly plateau at about $O(10^{-16})$ level and exhibit random fluctuations. Figure \ref{fig:3_1} shows the times (in seconds) required to run the FGIM and the SA-FGIM methods for certain ranges of the parameters $N$ and $M$. The surfaces associated with the FGIM have relatively much lower time values compared to their counterparts associated with the SA-FGIM. In all runs, the FGIM finished in under 0.01 seconds. In contrast, the SA-FGIM method exhibited substantially longer execution times, peaking at approximately 0.17 seconds when $M = N = 100$. At this point, the FGIM demonstrated a remarkable speed advantage, operating roughly 21 times faster than the SA-FGIM! This indicates that the FGIM is much more efficient in terms of execution time. As the number of cores in the computing device increases, the time gap between both methods is expected to widen significantly, especially when processing larger datasets. Table \ref{tab:1} shows a comprehensive comparisons with the methods of \citet{jena2021computational}, \citet{jena2023one}, and \citet{demir2012numerical}; the latter work uses a collocation method with cubic B-spline finite elements. Compared with all three methods, the FGIM achieves significantly smaller errors, even when employing relatively coarse mesh grids. For instance, the method in \cite{jena2023one} necessitates a mesh grid comprising $3,210,321$ points to achieve an $L_{\infty}$-error on the order of $10^{-12}$. In contrast, the FGIM attains the same level of accuracy with as few as 36 points and achieves near-full precision using just 44 mesh points-- a reduction in the mesh grid size by approximately 99.999\% in either case! The SA-FGIM outperforms all methods in terms of accuracy, scoring full machine precision using just 5 truncated Fourier series terms. Figure \ref{fig:4_1} displays the values of the collocation matrix condition number for certain ranges of the parameters $N$ and $M$. Observe how the \Kappa$\left({}_T\F{A}^{(n,M)}\right)$ remains relatively close from 1 at $n = 1$, and slightly improves for increasing matrix size, as discussed earlier in Section \ref{sec:ESA1}. 

\begin{figure*}[t]
\centering
\includegraphics[scale=0.35]{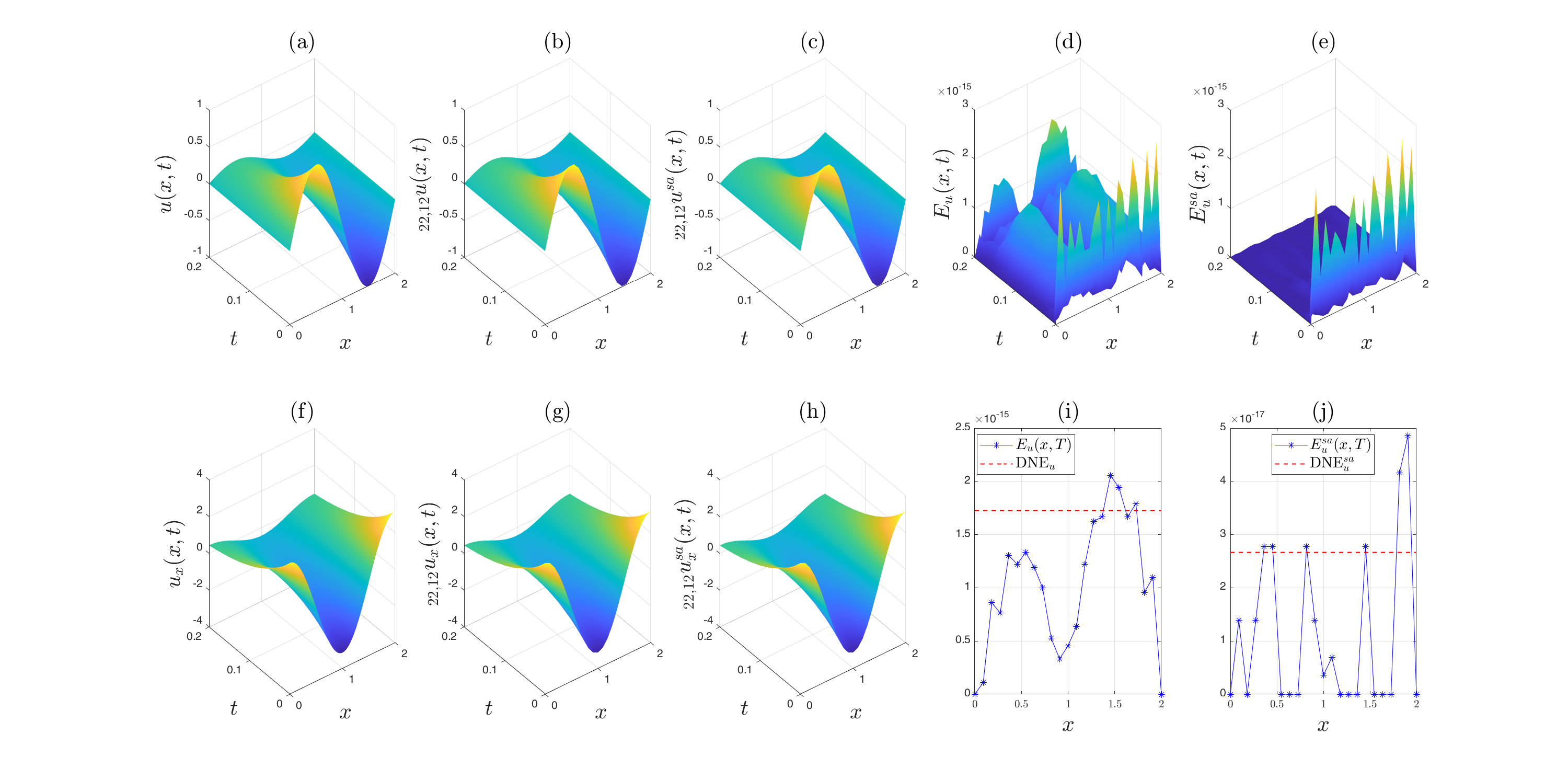}
\caption{Figs. (a)-(h) illustrate the surfaces of the exact and approximate solutions, their spatial derivatives and approximations, and their associated errors for Test Problem 1 using $M = 12, N = 22, N_0 = 24$, and $\lambda = -0.4$. Figs. (i) and (j) depict the absolute error profiles and the discrete norm level at the terminal time, $t = 0.2$.}
\label{fig:1_1}
\end{figure*}

\begin{figure}[t]
\centering
\includegraphics[scale=0.6]{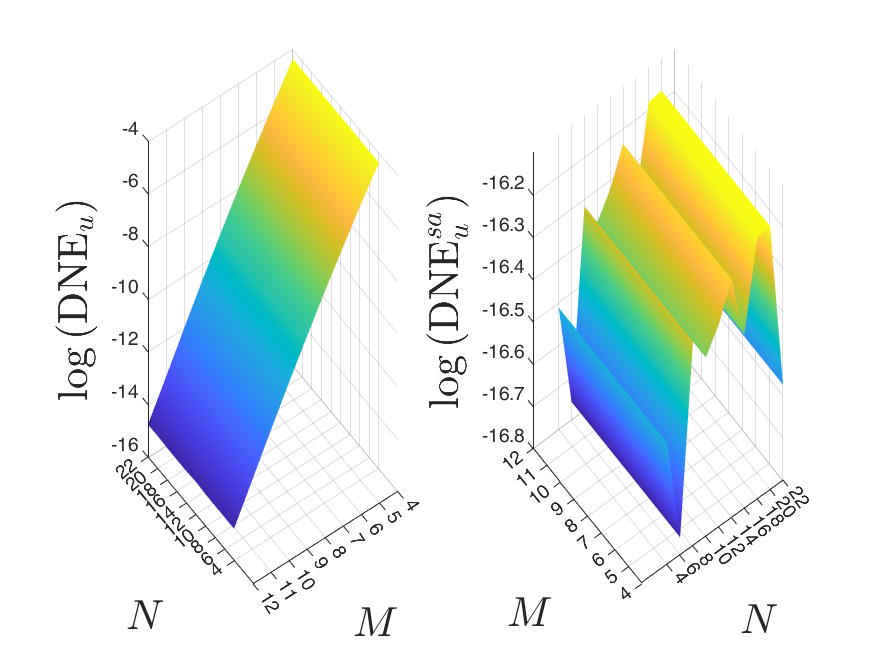}
\caption{The surfaces of the logarithms of the discrete error norms for the FGIM and SA-FGIM methods at the terminal time, $t = 0.2$. The computations were performed for $N = 4:2:22, N_0 = N + 2, M = 4:12$, and $\lambda = -0.4$.}
\label{fig:2_1}
\end{figure}

\begin{figure}[t]
\centering
\includegraphics[scale=0.6]{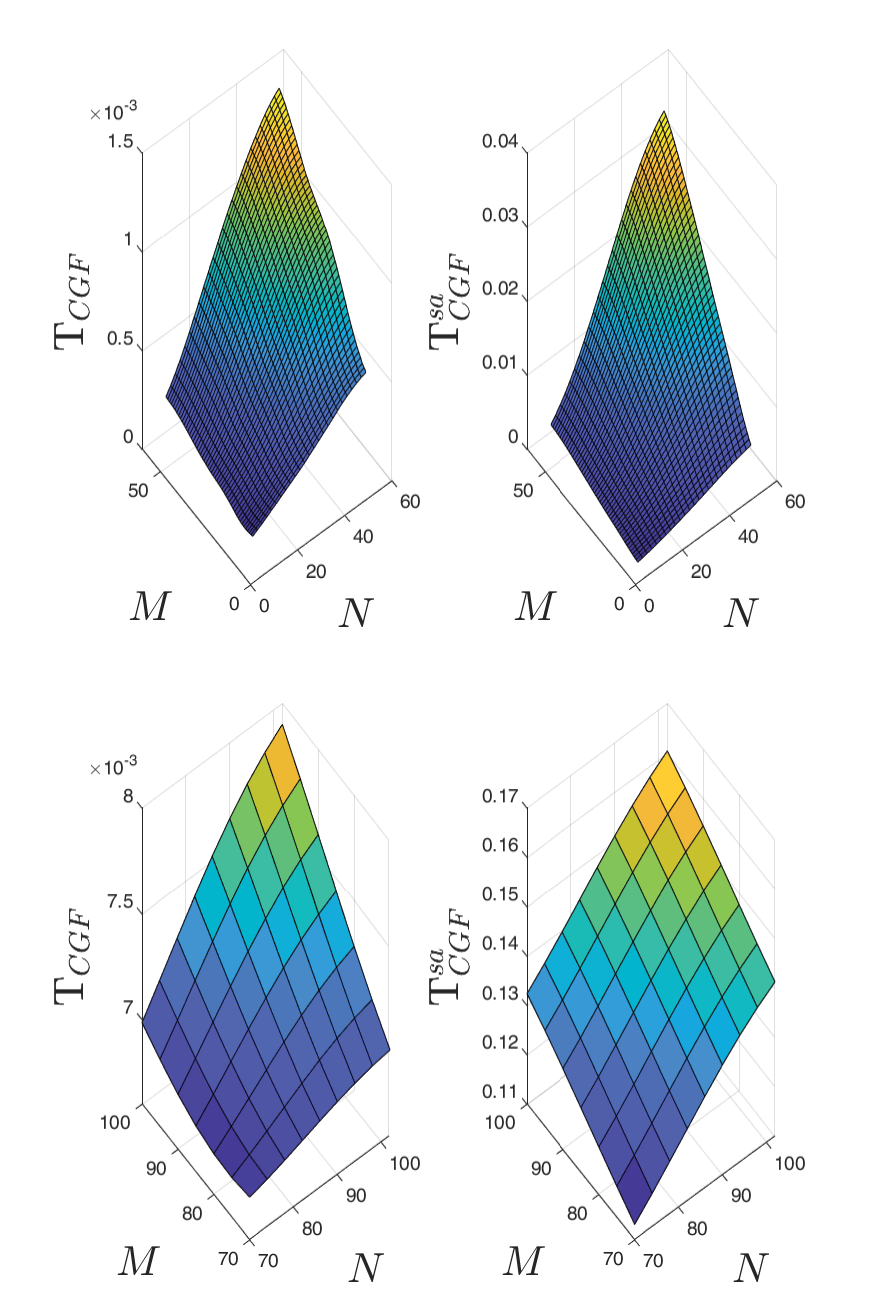}
\caption{The surfaces depict the median execution times (in seconds) of the FGIM and the SA-FGIM, averaged over multiple runs, for $\lambda = -0.4$, $N = 4:2:52$, $M = 4:52$ (upper row), and $N = 70:4:102$ and $M = 70:5:100$ (lower row). $N_0 = N + 2$ in all cases. Gaussian filtering was applied for smoothing the surfaces. The upper left surface was computed without parallel computing, while the lower left utilizes parallel processing.}
\label{fig:3_1}
\end{figure}

\begin{figure}[t]
\centering
\includegraphics[scale=0.6]{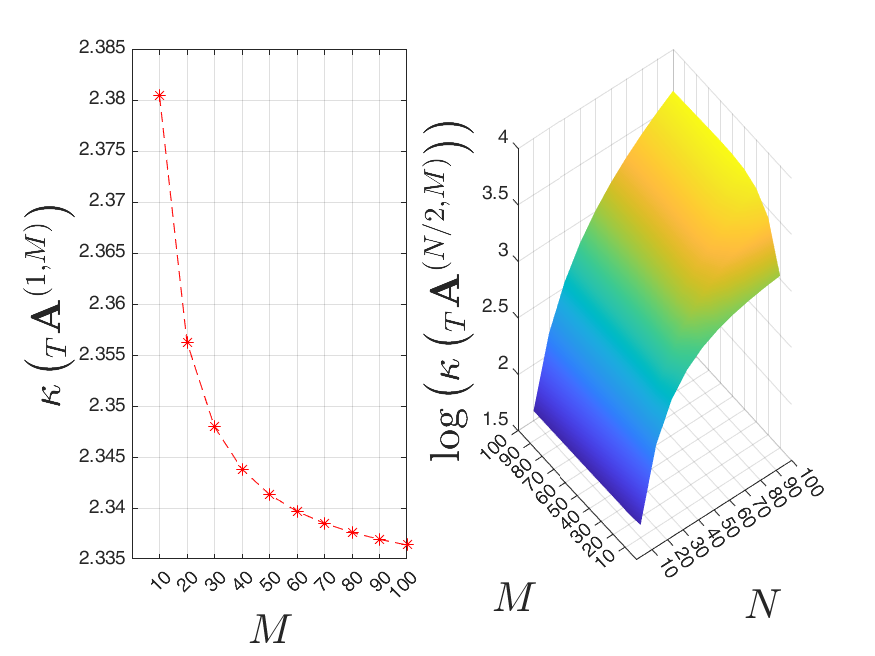}
\caption{The condition number curve (left) and its logarithmic surface (right) of the collocation matrix of System \eqref{eq:nnmm3Hi2_2h1}, associated with the FGIM at the fundamental and Nyquist frequencies. The plots were generated using the parameters $\lambda = -0.4$, $N = M = 10:10:100$. $N_0 = N + 2$ in all cases.}
\label{fig:4_1}
\end{figure}

\begin{table}[t]
\centering
\caption{Comparisons for Test Problem 1 at $t = 0.1$ and $\nu = 1$. The errors are rounded to 5 significant digits.}
\resizebox{0.4\textwidth}{!}{  
\begin{tabular}{ccc}  
\toprule
Method & Parameters & $L_{\infty}$ \\
\midrule
\cite{demir2012numerical} & $\Delta t = 0.00001, h = 0.00625$ & 2.3781e-05\\ 
\cite{jena2021computational} & $\Delta t = 0.00001, h = 0.00625$ & 3.1585e-10\\
\cite{jena2023one} & $\Delta t = 0.00001, h = 0.00625$ & 8.8459e-12 \\
FGIM & \multrow{$N=4, N_0 = 6, M=8, \lambda = -0.4$\\$N=4, N_0 = 6, M=10, \lambda = -0.4$} & \multrow{1.6445e-12\\7.21645e-16}\\
SA-FGIM & $N=4, N_0 = 6$ & 5.5511e-17\\
\bottomrule
\end{tabular}
}
\label{tab:1}
\end{table}

\textbf{Test Problem 2.} Consider Problem S with $\mu = 0, \nu = 1/\pi^2$, and $L = 2$ with initial and boundary functions $u_0(x) = \sin(\pi x)$ and $g(t) = 0$. The exact solution of this problem is $u(x,t) = e^{-t} \sin(\pi x)$ \cite{mohebbi2010high}. Figure \ref{fig:2_1hey} shows the surfaces of the exact and approximate solutions and their errors obtained by the FGIG and the SA-FGIG methods for some parameter values. Table \ref{tab:2_1} shows a comparison between the proposed methods and the CN and CBVM methods of \citet{sun2003high} as well as the method of \citet{mohebbi2010high}. The former methods combine fourth-order boundary value methods (BVMs) for time discretization with a fourth-order compact difference scheme for spatial discretization. The latter method utilizes a compact FD approximation of fourth-order accuracy to discretize spatial derivatives, and the resulting system of ordinary differential equations is then solved using the cubic $C^1$-spline collocation method. Both FGIM and SA-FGIM outperform these methods in terms of accuracy and computational cost, as shown in the table. In particular, the SA-FGIM exhibits the lowest $L_{\infty}$-error using $5$ terms of Fourier truncated series. The FGIM method demonstrates a significant improvement in accuracy as the number of temporal mesh points increases. Notably, adding only 3 temporal mesh points to the FGIM configuration leads to a remarkable decrease in error by approximately 5 orders of magnitude.

\begin{figure*}[t]
\centering
\includegraphics[scale=0.35]{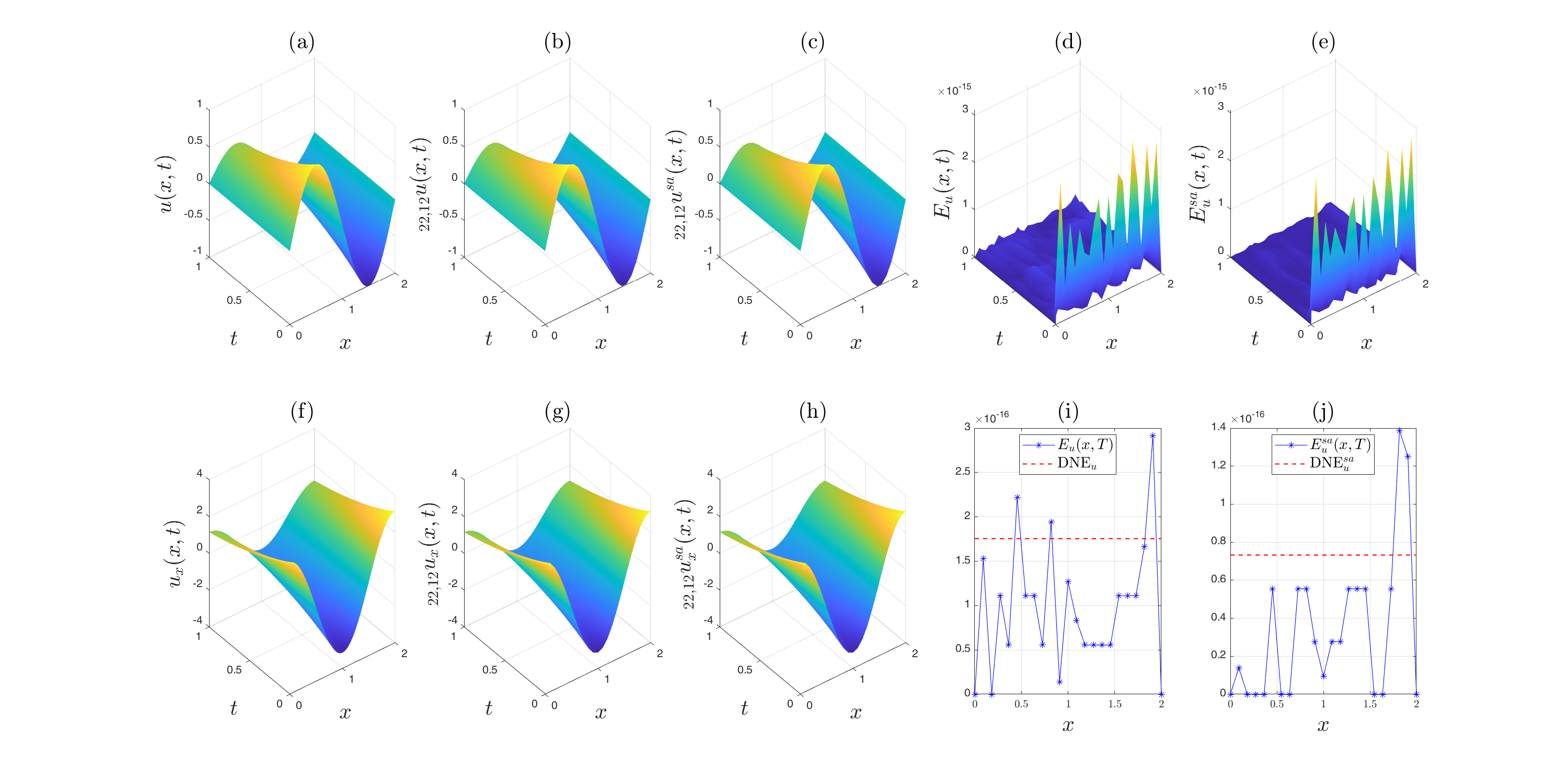}
\caption{Figs. (a)-(h) illustrate the surfaces of the exact and approximate solutions, their spatial derivatives and approximations, and their associated errors for Test Problem 2 using $M = 12, N = 22, N_0 = 24$, and $\lambda = -0.4$. Figs. (i) and (j) depict the absolute error profiles and the discrete norm level at the terminal time, $t = 1$.}
\label{fig:2_1hey}
\end{figure*}

\begin{table}[t]
\centering
\caption{Comparisons for Test Problem 2 at $t = 1$ and $\nu = 1/\pi^2$. The errors are rounded to 5 significant digits.}
\resizebox{0.4\textwidth}{!}{  
\begin{tabular}{ccc}  
\toprule
Method & Parameters & $L_{\infty}$ \\
\midrule
CN \cite{sun2003high} & $h = k = 0.00625$ & 1.1e-5\\
CBVM \cite{sun2003high} & $h = k = 0.00625$ & 2.3e-10\\
\cite{mohebbi2010high} & $h = k = 0.00625$ & 2.3436e-10\\ 
FGIM & \multrow{$N=4, N_0 = 6, M=7, \lambda = -0.4$\\$N=4, N_0 = 6, M=10, \lambda = -0.4$} & \multrow{7.0965e-11\\8.3267e-16}\\
SA-FGIM & $N=4, N_0 = 6$ & 5.5511e-17\\
\bottomrule
\end{tabular}
}
\label{tab:2_1}
\end{table}

\textbf{Test Problem 3.} Consider Problem S with $\mu = 0.01, \nu = 0.1$, and $L = 2$ with initial and boundary functions $u_0(x) = \sin(2 \pi x/L)$ and $g(t) = -e^{-\pi^2 \nu t} \sin(2 \pi \mu t/L)$. The exact solution of this problem is $u(x,t) = -e^{-\pi^2 \nu t} \sin[\pi (\mu t- 2 x/L)]$. Figure \ref{fig:Fig3_1hey} shows the surfaces of the exact and approximate solutions and their errors obtained by the FGIG and the SA-FGIG methods for some parameter values. Both methods resolve the solution surface well within an error of $O\left(10^{-3}\right)$ using as low as $16$ Fourier series terms and 5 time mesh points.

\begin{figure*}[t]
\centering
\includegraphics[scale=0.35]{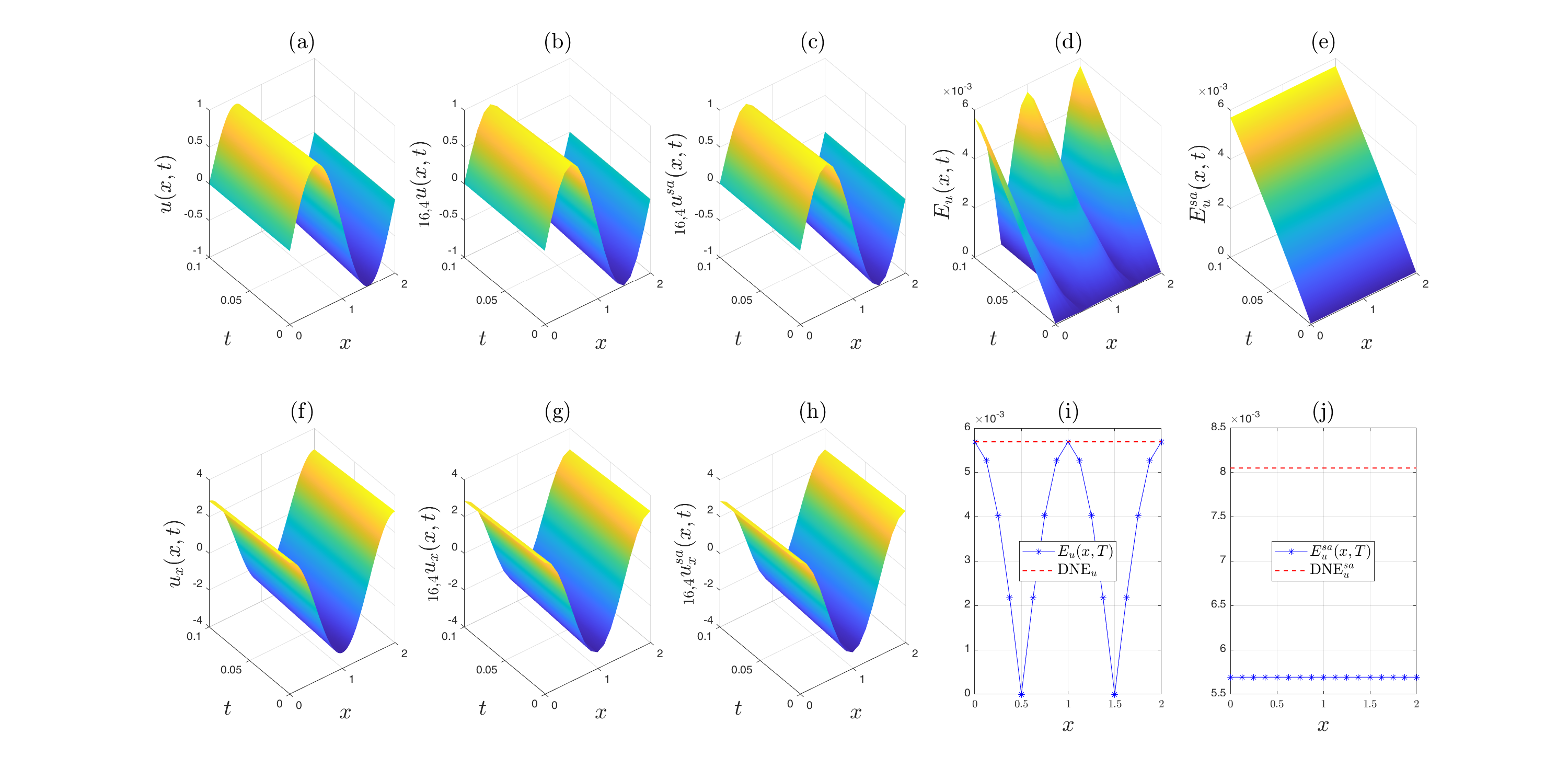}
\caption{Figs. (a)-(h) illustrate the surfaces of the exact and approximate solutions, their spatial derivatives and approximations, and their associated errors for Test Problem 3 using $M = 4, N = 16, N_0 = 18$, and $\lambda = -0.4$. Figs. (i) and (j) depict the absolute error profiles and the discrete norm level at the terminal time, $t = 0.1$.}
\label{fig:Fig3_1hey}
\end{figure*}

\section{Conclusion and Discussion}
\label{sec:Conc}
In this work, we introduced the FGIG method, a novel approach for solving the 1D AD equation with periodic boundary conditions. The method integrates Fourier series and Gegenbauer polynomials within a Galerkin framework. The FGIG method has proven particularly effective in resolving the intricacies of diffusion-dominated transport phenomena. Unlike traditional numerical methods, the FGIG method eliminates the need for iterative time-stepping by solving integral equations directly, significantly reducing computational overhead and temporal errors, and capturing the smooth and spreading behavior characteristic of diffusion-dominated regimes with excellent accuracy. Numerical experiments confirmed exponential convergence for smooth solutions and demonstrated stability under oscillatory conditions, particularly for Gegenbauer parameters within specified ranges. The FGIG method's inherent parallelizability enables efficient computation across multi-core architectures, making it suitable for large-scale simulations. Furthermore, the barycentric formulation of SGG quadrature ensures high precision in integral evaluations, a feature not commonly available in other numerical methods. This novel framework also facilitates the derivation of semi-analytical solutions, providing a benchmark for verifying numerical results. The FGIG method establishes a robust foundation for advancing numerical solutions of partial differential equations in computational science and engineering.

\section{Limitations of the FGIG Method and Future Works}
\label{sec:LCM1}
The FGIG method has demonstrated exceptional performance in diffusion-dominated problems, where solutions exhibit smooth and spreading behavior. However, its application to advection-dominated problems, characterized by steep gradients and sharp features, may be limited. In particular, the time evolution of the Fourier coefficients can become complex and difficult to predict in this case, especially for high P\'{e}clet numbers. This is because the advection term can introduce significant phase shifts and amplitude variations in the coefficients. In such scenarios, the current method often cannot accurately capture the sharp features of the solution. To address these limitations, future research could integrate adaptive mesh refinement strategies. This approach would strategically concentrate computational resources in regions exhibiting high gradients, thereby improving the method's ability to resolve sharp features and accurately simulate advection-dominated phenomena. Another promising avenue is extending the FGIG framework to higher-dimensional problems and coupled systems, broadening its applicability. Exploring hybrid approaches that combine the FGIG method with possibly machine learning techniques offers another exciting direction. These methods could use data-driven insights to improve computational efficiency and accuracy for complex real-world problems.

\section*{Declarations}
\subsection*{Competing Interests}
The author declares there is no conflict of interests.

\subsection*{Availability of Supporting Data}
The author declares that the data supporting the findings of this study are available within the article.

\subsection*{Ethical Approval and Consent to Participate and Publish}
Not Applicable.

\subsection*{Human and Animal Ethics}
Not Applicable.

\subsection*{Consent for Publication}
Not Applicable.

\subsection*{Funding}
The author received no financial support for the research, authorship, and/or publication of this article.

\subsection*{Authors' Contributions}
The author confirms sole responsibility for the following: study conception and design, data collection, analysis and interpretation of results, and manuscript preparation.


\appendix
\section{Mathematical Proofs}
\label{sec:MP1}
The following theorem establishes the $L^2$ nature of $\C{I}_x^{(x)} \psi + g$, provided that $\psi \in L^2$. 
\begin{thm}\label{thm:1}
Let the SSD of Problem S, $\psi \in L^2(\FOmega_{L \times T})$. Then $u = \C{I}_x^{(x)} \psi + g \in L^2(\FOmega_{L \times T})\;\forall \{L, T\} \subset \MBRP$. 
\end{thm}
\begin{proof}
First, notice that 
\begin{equation}\label{eq:Nice0}
\C{I}_L^{(x)}\C{I}_T^{(t)}|u{|^2} = \C{I}_L^{(x)}\C{I}_T^{(t)}\left( {(\C{I}_x^{(x)}\psi)^2 + 2 \C{I}_x^{(x)}\psi g + g^2} \right).
\end{equation}
Applying Cauchy-Schwarz inequality on the first integrand term gives
\begin{equation}\label{eq:Nice1}
|\C{I}_x^{(x)}\psi {|^2} \le x\,\C{I}_x^{(x)}|\psi {|^2} \le L\,\C{I}_x^{(x)}|\psi {|^2} < \infty,\quad \forall L \in \MBR^+.
\end{equation}
Using Fubini's theorem to swap the order of integration gives
\begin{gather*}
\C{I}_L^{(x)}\C{I}_T^{(t)} {|\C{I}_x^{(x)}\psi|^2} \le \C{I}_L^{(x)}\C{I}_T^{(t)}\left( {x\,\C{I}_x^{(x)}|\psi {|^2}} \right) = \C{I}_T^{(t)}\C{I}_L^{(\xi )}\left( {\left[ {\C{I}_{\xi ,L}^{(x)}x} \right]\,|\psi {|^2}} \right)\\
 = \C{I}_T^{(t)}\C{I}_L^{(\xi )}\left( {\left[ {\frac{{{L^2} - {\xi ^2}}}{2}} \right]\,|\psi {|^2}} \right) < \infty,
\end{gather*}
since $\psi \in L^2(\FOmega_{L \times T})$, by assumption. The inequality
\[\C{I}_L^{(x)}\C{I}_T^{(t)}\left( {2\C{I}_x^{(x)}\psi g + g^2} \right) < \infty,\]
follows readily by Ineq. \eqref{eq:Nice1} and the boundedness of $g$ on $\FOmega_T$. This completes the proof of the theorem.
\end{proof}

The following theorem proves the $\beta$-analyticity of the exact solution of Test Problem 1.

\begin{thm}\label{thm:MP1}
The function $u(x,t) = e^{-\pi^2 t} \sin(\pi x)$ is $\beta$-analytic in $x\,\forall (x,t) \in \FOmega_{L \times T}: L < 2$. 
\end{thm}
\begin{proof}
First, notice that $u$ is analytic on ${\F{C}_{L,\infty}}\,\forall t \in \FOmega_T$, because the sine function is entire. 
Observe also that
\begin{gather}
u(x+iy) = e^{-\pi^2 t} \sin(\pi (x+iy))\\
= e^{-\pi^2 t} \left[\sin(\pi x) \cosh(\pi y) + i \cos(\pi x) \sinh(\pi y)\right],
\end{gather}
so
\begin{equation}
|u(x+iy)| = e^{-\pi^2 t} \cosh(\pi y),
\end{equation}
using the identity $\cosh^2(\pi y) - \sinh^{2}(\pi y) = 1$. On ${\F{C}_{L,\beta}}$, the supremum occurs at $y = \pm \beta$, so
\begin{equation}
{\left\|u\right\|_{{\C{A}_{L,\beta}}}} = e^{-\pi^2 t} \cosh(\pi \beta).
\end{equation}
Consequently,
\begin{equation}
{\lim _{\beta  \to \infty }}\frac{{{{\left\| u \right\|}_{{\C{A}_{T,\beta} }}}}}{{{e^{{\omega _\beta }}}}} = e^{-\pi^2 t} {\lim _{\beta  \to \infty }}\frac{\cosh(\pi \beta)}{{{e^{{2 \pi \beta/L}}}}} \sim e^{-\pi^2 t} {\lim _{\beta  \to \infty }}\frac{\frac{1}{2} e^{\pi \beta}}{{{e^{{2 \pi \beta/L}}}}} = 0,\quad \forall L < 2.
\end{equation}
This completes the proof of the theorem.
\end{proof}

The following theorem highlights the existence of relatively small singular values of the GIM $\F{Q}$, as $\lambda \to -0.5$.

\begin{thm}\label{thm:Woowmmm1}
There exists at least one near zero singular value of $\F{Q}$, not necessarily the smallest singular value, when $\lambda \to -0.5$.
\end{thm}
\begin{proof}
Since $\F{Q}$ has at least one eigenvalue $\lambdabar_{\F{Q}} \approx 0$, as $\lambda \to -0.5$, then there exists a nonzero vector $\bmv: \F{Q} \bmv \approx \bmzer$. Therefore,
\begin{equation}
\F{Q}^{\top} \F{Q} \bmv = \F{Q}^{\top} (\F{Q} \bmv) \approx \F{Q}^{\top} \bmzer \approx \bmzer.
\end{equation}
This implies that $\bmv$ is approximately an eigenvector of $\F{Q}^{\top} \F{Q}$ with an eigenvalue $\lambdabar_{\F{Q}^{\top} \F{Q}} \approx 0$, and a corresponding singular value, $\sigma$, given by
\begin{equation}
\sigma = \sqrt{\lambdabar_{\F{Q}^{\top} \F{Q}}} \approx 0,
\end{equation} 
from which the proof is accomplished.
\end{proof}

The following theorem highlights the peak of \Kappa$\left({}_T\F{A}^{(n,M)}\right)$ among all frequencies, assuming that ${}_T\F{Q}$ is diagonalizable, which is often the case as observed by extensive numerical experiments.

\begin{thm}\label{thm:poyopi1}
Assume that ${}_T\F{Q}$ is a diagonalizable matrix. If $\mu$ and $\nu$ do not vanish simultaneously, then 
\begin{equation}
\max\limits_{n \in \MBN_{N/2}}\text{\Kappa}\left({}_T\F{A}^{(n,M)}\right) = \text{\Kappa}\left({}_T\F{A}^{(N/2,M)}\right).
\end{equation}
\end{thm} 
\begin{proof}
Since ${}_T\F{Q}$ is a diagonalizable matrix, we can express it as 
\begin{equation}\label{eq:dssdgsgs1}
{}_T\F{Q} = \F{V} \F{\Lambda} \F{V}^{-1}, 
\end{equation}
where $\F{V}$ is a matrix of eigenvectors of ${}_T\F{Q}$ and $\F{\Lambda}$ is a diagonal matrix containing the eigenvalues of ${}_T\F{Q}$. Substituting \eqref{eq:dssdgsgs1} in the collocation matrix gives 
\begin{equation}
{}_T\F{A}^{(n,M)} = \F{I}_{M+1} + \alpha_n\;\F{V} \F{\Lambda} \F{V}^{-1} = \F{V}\left(\F{I}_{M+1} + \alpha_n\; \F{\Lambda}\right) \F{V}^{-1}.
\end{equation}
The eigenvalues of ${}_T\F{A}^{(n,M)}$ are given by the diagonal elements of $\F{I}_{M+1} + \alpha_n\; \F{\Lambda}$, which are $1 + \alpha_n \lambdabar_{0:M}$. Therefore,
\begin{equation}
\text{\Kappa}\left({}_T\F{A}^{(n,M)}\right) = \text{\Kappa}\left(\F{V}\right) \cdot \frac{1 + \alpha_n \lambdabar_{\max}}{1 + \alpha_n \lambdabar_{\min}}.
\end{equation}
As $\alpha_n$ increases, the eigenvalues of ${}_T\F{A}^{(n,M)}$ become more spread out.
This increased spread in eigenvalues generally leads to a larger ratio between $1 + \alpha_n \lambdabar_{\max}$ and $1 + \alpha_n \lambdabar_{\min}$, and, consequently, a higher condition number. The proof is accomplished by realizing that $\max_{n\in \MBN_{N/2}}\alpha_n = \alpha_{N/2}$.
\end{proof}
Theorem \ref{thm:poyopi1} and the analysis conducted in Section \ref{sec:ESA1} reveal further that $\text{\Kappa}\left({}_T\F{A}^{(n,M)}\right) \to \infty\,\foralll N, \lambda$. 

\bibliographystyle{model1-num-names}
\bibliography{Bib}

\begin{thebibliography}{62}
\expandafter\ifx\csname natexlab\endcsname\relax\def\natexlab#1{#1}\fi
\providecommand{\bibinfo}[2]{#2}
\ifx\xfnm\relax \def\xfnm[#1]{\unskip,\space#1}\fi
\bibitem[{Appadu et~al.(2017)Appadu, Lubuma, and
  Mphephu}]{appadu2017computational}
\bibinfo{author}{A.~R. Appadu}, \bibinfo{author}{J.~M. Lubuma},
  \bibinfo{author}{N.~Mphephu},
\newblock \bibinfo{title}{Computational study of three numerical methods for
  some linear and nonlinear advection-diffusion-reaction problems},
\newblock \bibinfo{journal}{Progress in Computational Fluid Dynamics, an
  International Journal} \bibinfo{volume}{17} (\bibinfo{year}{2017})
  \bibinfo{pages}{114--129}.
\bibitem[{Khalsaraei and Jahandizi(2017)}]{khalsaraei2017efficient}
\bibinfo{author}{M.~M. Khalsaraei}, \bibinfo{author}{R.~S. Jahandizi},
\newblock \bibinfo{title}{Efficient explicit nonstandard finite difference
  scheme with positivity-preserving property},
\newblock \bibinfo{journal}{Gazi University Journal of Science}
  \bibinfo{volume}{30} (\bibinfo{year}{2017}) \bibinfo{pages}{259--268}.
\bibitem[{Solis and Gonzalez(2017)}]{solis2017numerical}
\bibinfo{author}{F.~J. Solis}, \bibinfo{author}{L.~M. Gonzalez},
\newblock \bibinfo{title}{A numerical approach for a model of the precancer
  lesions caused by the human papillomavirus},
\newblock \bibinfo{journal}{Journal of Difference Equations and Applications}
  \bibinfo{volume}{23} (\bibinfo{year}{2017}) \bibinfo{pages}{1093--1104}.
\bibitem[{Al-khafaji and Al-Zubaidi(2024)}]{al2024numerical}
\bibinfo{author}{F.~Al-khafaji}, \bibinfo{author}{H.~A. Al-Zubaidi},
\newblock \bibinfo{title}{Numerical modeling of instantaneous spills in
  one-dimensional river systems},
\newblock \bibinfo{journal}{Nature Environment \& Pollution Technology}
  \bibinfo{volume}{23} (\bibinfo{year}{2024}).
\bibitem[{Cerfontaine et~al.(2016)Cerfontaine, Radioti, Collin, and
  Charlier}]{cerfontaine2016formulation}
\bibinfo{author}{B.~Cerfontaine}, \bibinfo{author}{G.~Radioti},
  \bibinfo{author}{F.~Collin}, \bibinfo{author}{R.~Charlier},
\newblock \bibinfo{title}{Formulation of a {1D} finite element of heat
  exchanger for accurate modelling of the grouting behaviour: {A}pplication to
  cyclic thermal loading},
\newblock \bibinfo{journal}{Renewable Energy} \bibinfo{volume}{96}
  (\bibinfo{year}{2016}) \bibinfo{pages}{65--79}.
\bibitem[{Wang and Yuan(2022)}]{wang2022discrete}
\bibinfo{author}{S.~Wang}, \bibinfo{author}{G.~Yuan},
\newblock \bibinfo{title}{Discrete strong extremum principles for finite
  element solutions of diffusion problems with nonlinear corrections},
\newblock \bibinfo{journal}{Applied Numerical Mathematics}
  \bibinfo{volume}{174} (\bibinfo{year}{2022}) \bibinfo{pages}{1--16}.
\bibitem[{Wang and Yuan(2024)}]{wang2024discrete}
\bibinfo{author}{S.~Wang}, \bibinfo{author}{G.~Yuan},
\newblock \bibinfo{title}{Discrete strong extremum principles for finite
  element solutions of advection-diffusion problems with nonlinear
  corrections},
\newblock \bibinfo{journal}{International Journal for Numerical Methods in
  Fluids} \bibinfo{volume}{96} (\bibinfo{year}{2024})
  \bibinfo{pages}{1990--2005}.
\bibitem[{Jena and Senapati(2023)}]{jena2023one}
\bibinfo{author}{S.~R. Jena}, \bibinfo{author}{A.~Senapati},
\newblock \bibinfo{title}{One-dimensional heat and advection-diffusion equation
  based on improvised cubic {B}-spline collocation, finite element method and
  {C}rank-{N}icolson technique},
\newblock \bibinfo{journal}{International Communications in Heat and Mass
  Transfer} \bibinfo{volume}{147} (\bibinfo{year}{2023})
  \bibinfo{pages}{106958}.
\bibitem[{Sejekan and Ollivier-Gooch(2016)}]{sejekan2016improving}
\bibinfo{author}{C.~B. Sejekan}, \bibinfo{author}{C.~F. Ollivier-Gooch},
\newblock \bibinfo{title}{Improving finite-volume diffusive fluxes through
  better reconstruction},
\newblock \bibinfo{journal}{Computers \& Fluids} \bibinfo{volume}{139}
  (\bibinfo{year}{2016}) \bibinfo{pages}{216--232}.
\bibitem[{Chernyshenko et~al.(2017)Chernyshenko, Olshahskii, and
  Vassilevski}]{chernyshenko2017hybrid}
\bibinfo{author}{A.~Chernyshenko}, \bibinfo{author}{M.~Olshahskii},
  \bibinfo{author}{Y.~Vassilevski},
\newblock \bibinfo{title}{A hybrid finite volume—finite element method for
  modeling flows in fractured media},
\newblock in: \bibinfo{booktitle}{Finite Volumes for Complex Applications
  VIII-Hyperbolic, Elliptic and Parabolic Problems: FVCA 8, Lille, France, June
  2017 8}, \bibinfo{organization}{Springer}, pp. \bibinfo{pages}{527--535}.
\bibitem[{Kramarenko et~al.(2017)Kramarenko, Nikitin, and
  Vassilevski}]{kramarenko2017nonlinear}
\bibinfo{author}{V.~Kramarenko}, \bibinfo{author}{K.~Nikitin},
  \bibinfo{author}{Y.~Vassilevski},
\newblock \bibinfo{title}{A nonlinear correction {FV} scheme for near-well
  regions},
\newblock in: \bibinfo{booktitle}{Finite Volumes for Complex Applications
  VIII-Hyperbolic, Elliptic and Parabolic Problems: FVCA 8, Lille, France, June
  2017 8}, \bibinfo{organization}{Springer}, pp. \bibinfo{pages}{507--516}.
\bibitem[{Hwang and Son(2023)}]{hwang2023efficient}
\bibinfo{author}{S.~Hwang}, \bibinfo{author}{S.~Son},
\newblock \bibinfo{title}{An efficient {HLL}-based scheme for capturing
  contact-discontinuity in scalar transport by shallow water flow},
\newblock \bibinfo{journal}{Communications in Nonlinear Science and Numerical
  Simulation} \bibinfo{volume}{127} (\bibinfo{year}{2023})
  \bibinfo{pages}{107531}.
\bibitem[{Mei et~al.(2024)Mei, Zhou, and Liu}]{mei2024unified}
\bibinfo{author}{D.~Mei}, \bibinfo{author}{K.~Zhou}, \bibinfo{author}{C.-H.
  Liu},
\newblock \bibinfo{title}{Unified finite-volume physics informed neural
  networks to solve the heterogeneous partial differential equations},
\newblock \bibinfo{journal}{Knowledge-Based Systems} \bibinfo{volume}{295}
  (\bibinfo{year}{2024}) \bibinfo{pages}{111831}.
\bibitem[{Bonaventura and Ferretti(2016)}]{bonaventura2016flux}
\bibinfo{author}{L.~Bonaventura}, \bibinfo{author}{R.~Ferretti},
\newblock \bibinfo{title}{Flux form semi-{L}agrangian methods for parabolic
  problems},
\newblock \bibinfo{journal}{Communications in Applied and Industrial
  Mathematics} \bibinfo{volume}{7} (\bibinfo{year}{2016})
  \bibinfo{pages}{56--73}.
\bibitem[{Bokanowski and Simarmata(2016)}]{bokanowski2016semi}
\bibinfo{author}{O.~Bokanowski}, \bibinfo{author}{G.~Simarmata},
\newblock \bibinfo{title}{Semi-{L}agrangian discontinuous {G}alerkin schemes
  for some first-and second-order partial differential equations},
\newblock \bibinfo{journal}{ESAIM: Mathematical Modelling and Numerical
  Analysis} \bibinfo{volume}{50} (\bibinfo{year}{2016})
  \bibinfo{pages}{1699--1730}.
\bibitem[{Bakhtiari et~al.(2016)Bakhtiari, Malhotra, Raoofy, Mehl, Bungartz,
  and Biros}]{bakhtiari2016parallel}
\bibinfo{author}{A.~Bakhtiari}, \bibinfo{author}{D.~Malhotra},
  \bibinfo{author}{A.~Raoofy}, \bibinfo{author}{M.~Mehl},
  \bibinfo{author}{H.-J. Bungartz}, \bibinfo{author}{G.~Biros},
\newblock \bibinfo{title}{A parallel arbitrary-order accurate amr algorithm for
  the scalar advection-diffusion equation},
\newblock in: \bibinfo{booktitle}{SC'16: Proceedings of the International
  Conference for High Performance Computing, Networking, Storage and Analysis},
  \bibinfo{organization}{IEEE}, pp. \bibinfo{pages}{514--525}.
\bibitem[{Liu et~al.(2025)Liu, Qiao, and Feng}]{liu2025semi}
\bibinfo{author}{Y.~Liu}, \bibinfo{author}{Y.~Qiao}, \bibinfo{author}{X.~Feng},
\newblock \bibinfo{title}{A semi-lagrangian radial basis function partition of
  unity closest point method for advection-diffusion equations on surfaces},
\newblock \bibinfo{journal}{Computers \& Mathematics with Applications}
  \bibinfo{volume}{177} (\bibinfo{year}{2025}) \bibinfo{pages}{100--114}.
\bibitem[{Bergmann et~al.(2022)Bergmann, Carlino, and
  Iollo}]{bergmann2022second}
\bibinfo{author}{M.~Bergmann}, \bibinfo{author}{M.~G. Carlino},
  \bibinfo{author}{A.~Iollo},
\newblock \bibinfo{title}{Second order {ADER} scheme for unsteady
  advection-diffusion on moving overset grids with a compact transmission
  condition},
\newblock \bibinfo{journal}{SIAM Journal on Scientific Computing}
  \bibinfo{volume}{44} (\bibinfo{year}{2022}) \bibinfo{pages}{A524--A553}.
\bibitem[{Sultan et~al.(2023)Sultan, Zhengce, and Usman}]{sultan2023stable}
\bibinfo{author}{S.~Sultan}, \bibinfo{author}{Z.~Zhengce},
  \bibinfo{author}{M.~Usman},
\newblock \bibinfo{title}{A stable r-adaptive mesh technique to analyze the
  advection-diffusion equation},
\newblock \bibinfo{journal}{Physica Scripta} \bibinfo{volume}{98}
  (\bibinfo{year}{2023}) \bibinfo{pages}{085212}.
\bibitem[{Terekhov et~al.(2023)Terekhov, Butakov, Danilov, and
  Vassilevski}]{terekhov2023dynamic}
\bibinfo{author}{K.~M. Terekhov}, \bibinfo{author}{I.~D. Butakov},
  \bibinfo{author}{A.~A. Danilov}, \bibinfo{author}{Y.~V. Vassilevski},
\newblock \bibinfo{title}{Dynamic adaptive moving mesh finite-volume method for
  the blood flow and coagulation modeling},
\newblock \bibinfo{journal}{International Journal for Numerical Methods in
  Biomedical Engineering} \bibinfo{volume}{39} (\bibinfo{year}{2023})
  \bibinfo{pages}{e3731}.
\bibitem[{Lee and Engquist(2016)}]{lee2016multiscale}
\bibinfo{author}{Y.~Lee}, \bibinfo{author}{B.~Engquist},
\newblock \bibinfo{title}{Multiscale numerical methods for passive
  advection--diffusion in incompressible turbulent flow fields},
\newblock \bibinfo{journal}{Journal of Computational Physics}
  \bibinfo{volume}{317} (\bibinfo{year}{2016}) \bibinfo{pages}{33--46}.
\bibitem[{Le~Bris et~al.(2017)Le~Bris, Legoll, and Madiot}]{le2017numerical}
\bibinfo{author}{C.~Le~Bris}, \bibinfo{author}{F.~Legoll},
  \bibinfo{author}{F.~Madiot},
\newblock \bibinfo{title}{A numerical comparison of some multiscale finite
  element approaches for advection-dominated problems in heterogeneous media},
\newblock \bibinfo{journal}{ESAIM: Mathematical Modelling and Numerical
  Analysis} \bibinfo{volume}{51} (\bibinfo{year}{2017})
  \bibinfo{pages}{851--888}.
\bibitem[{Wang et~al.(2024)Wang, Peng, Al~Mahbub, and
  Zheng}]{wang2024partitioned}
\bibinfo{author}{Y.~Wang}, \bibinfo{author}{Z.~Peng}, \bibinfo{author}{M.~A.
  Al~Mahbub}, \bibinfo{author}{H.~Zheng},
\newblock \bibinfo{title}{Partitioned schemes for the blood solute dynamics
  model by the variational multiscale method},
\newblock \bibinfo{journal}{Applied Numerical Mathematics}
  \bibinfo{volume}{198} (\bibinfo{year}{2024}) \bibinfo{pages}{318--345}.
\bibitem[{Le~Bris et~al.(2016)Le~Bris, Legoll, and
  Madiot}]{le2016stabilization}
\bibinfo{author}{C.~Le~Bris}, \bibinfo{author}{F.~Legoll},
  \bibinfo{author}{F.~Madiot},
\newblock \bibinfo{title}{Stabilization of non-coercive problems using the
  invariant measure},
\newblock \bibinfo{journal}{COMPTES RENDUS MATHEMATIQUE} \bibinfo{volume}{354}
  (\bibinfo{year}{2016}) \bibinfo{pages}{799--803}.
\bibitem[{Benedetto et~al.(2016)Benedetto, Berrone, Borio, Pieraccini, and
  Scialo}]{benedetto2016order}
\bibinfo{author}{M.~F. Benedetto}, \bibinfo{author}{S.~Berrone},
  \bibinfo{author}{A.~Borio}, \bibinfo{author}{S.~Pieraccini},
  \bibinfo{author}{S.~Scialo},
\newblock \bibinfo{title}{Order preserving {SUPG} stabilization for the virtual
  element formulation of advection--diffusion problems},
\newblock \bibinfo{journal}{Computer Methods in Applied Mechanics and
  Engineering} \bibinfo{volume}{311} (\bibinfo{year}{2016})
  \bibinfo{pages}{18--40}.
\bibitem[{Gonz{\'a}lez-Pinto et~al.(2017)Gonz{\'a}lez-Pinto,
  Hern{\'a}ndez-Abreu, and P{\'e}rez-Rodr{\'\i}guez}]{gonzalez2017w}
\bibinfo{author}{S.~Gonz{\'a}lez-Pinto},
  \bibinfo{author}{D.~Hern{\'a}ndez-Abreu},
  \bibinfo{author}{S.~P{\'e}rez-Rodr{\'\i}guez},
\newblock \bibinfo{title}{W-methods to stabilize standard explicit
  {R}unge--{K}utta methods in the time integration of
  advection--diffusion--reaction {PDEs}},
\newblock \bibinfo{journal}{Journal of Computational and Applied Mathematics}
  \bibinfo{volume}{316} (\bibinfo{year}{2017}) \bibinfo{pages}{143--160}.
\bibitem[{Biezemans et~al.(2025)Biezemans, Le~Bris, Legoll, and
  Lozinski}]{biezemans2025msfem}
\bibinfo{author}{R.~A. Biezemans}, \bibinfo{author}{C.~Le~Bris},
  \bibinfo{author}{F.~Legoll}, \bibinfo{author}{A.~Lozinski},
\newblock \bibinfo{title}{{MsFEM} for advection-dominated problems in
  heterogeneous media: {S}tabilization via nonconforming variants},
\newblock \bibinfo{journal}{Computer Methods in Applied Mechanics and
  Engineering} \bibinfo{volume}{433} (\bibinfo{year}{2025})
  \bibinfo{pages}{117496}.
\bibitem[{Cardone et~al.(2017)Cardone, D'Ambrosio, and
  Paternoster}]{cardone2017exponentially}
\bibinfo{author}{A.~Cardone}, \bibinfo{author}{R.~D'Ambrosio},
  \bibinfo{author}{B.~Paternoster},
\newblock \bibinfo{title}{Exponentially fitted {IMEX} methods for
  advection--diffusion problems},
\newblock \bibinfo{journal}{Journal of Computational and Applied Mathematics}
  \bibinfo{volume}{316} (\bibinfo{year}{2017}) \bibinfo{pages}{100--108}.
\bibitem[{Laouar et~al.(2023)Laouar, Arar, and Talaat}]{laouar2023efficient}
\bibinfo{author}{Z.~Laouar}, \bibinfo{author}{N.~Arar},
  \bibinfo{author}{A.~Talaat},
\newblock \bibinfo{title}{Efficient spectral legendre {G}alerkin approach for
  the advection diffusion equation with constant and variable coefficients
  under mixed {R}obin boundary conditions},
\newblock \bibinfo{journal}{Advances in the Theory of Nonlinear Analysis and
  its Application} \bibinfo{volume}{7} (\bibinfo{year}{2023})
  \bibinfo{pages}{133--147}.
\bibitem[{Takagi et~al.(2024)Takagi, Moriya, Aoki, Endo, Muramatsu, and
  Fukagata}]{takagi2024implementation}
\bibinfo{author}{K.~Takagi}, \bibinfo{author}{N.~Moriya},
  \bibinfo{author}{S.~Aoki}, \bibinfo{author}{K.~Endo},
  \bibinfo{author}{M.~Muramatsu}, \bibinfo{author}{K.~Fukagata},
\newblock \bibinfo{title}{Implementation of spectral methods on ising machines:
  toward flow simulations on quantum annealers},
\newblock \bibinfo{journal}{Fluid Dynamics Research} \bibinfo{volume}{56}
  (\bibinfo{year}{2024}) \bibinfo{pages}{061401}.
\bibitem[{Korkmaz and Akmaz(2018)}]{korkmaz2018numerical}
\bibinfo{author}{A.~Korkmaz}, \bibinfo{author}{H.~K. Akmaz},
\newblock \bibinfo{title}{Numerical solution of non-conservative linear
  transport problems},
\newblock \bibinfo{journal}{TWMS Journal of Applied and Engineering
  Mathematics} \bibinfo{volume}{8} (\bibinfo{year}{2018})
  \bibinfo{pages}{167--177}.
\bibitem[{Jena and Gebremedhin(2021)}]{jena2021computational}
\bibinfo{author}{S.~R. Jena}, \bibinfo{author}{G.~S. Gebremedhin},
\newblock \bibinfo{title}{Computational technique for heat and
  advection--diffusion equations},
\newblock \bibinfo{journal}{Soft Computing} \bibinfo{volume}{25}
  (\bibinfo{year}{2021}) \bibinfo{pages}{11139--11150}.
\bibitem[{Palav and Pradhan(2025)}]{palav2025redefined}
\bibinfo{author}{M.~S. Palav}, \bibinfo{author}{V.~H. Pradhan},
\newblock \bibinfo{title}{Redefined fourth order uniform hyperbolic polynomial
  {B}-splines based collocation method for solving advection-diffusion
  equation},
\newblock \bibinfo{journal}{Applied Mathematics and Computation}
  \bibinfo{volume}{484} (\bibinfo{year}{2025}) \bibinfo{pages}{128992}.
\bibitem[{Gharehbaghi(2016)}]{gharehbaghi2016explicit}
\bibinfo{author}{A.~Gharehbaghi},
\newblock \bibinfo{title}{Explicit and implicit forms of differential
  quadrature method for advection--diffusion equation with variable
  coefficients in semi-infinite domain},
\newblock \bibinfo{journal}{Journal of Hydrology} \bibinfo{volume}{541}
  (\bibinfo{year}{2016}) \bibinfo{pages}{935--940}.
\bibitem[{Ilati and Dehghan(2016)}]{ilati2016remediation}
\bibinfo{author}{M.~Ilati}, \bibinfo{author}{M.~Dehghan},
\newblock \bibinfo{title}{Remediation of contaminated groundwater by meshless
  local weak forms},
\newblock \bibinfo{journal}{Computers \& Mathematics with Applications}
  \bibinfo{volume}{72} (\bibinfo{year}{2016}) \bibinfo{pages}{2408--2416}.
\bibitem[{Suchde et~al.(2017)Suchde, Kuhnert, Schr{\"o}der, and
  Klar}]{suchde2017flux}
\bibinfo{author}{P.~Suchde}, \bibinfo{author}{J.~Kuhnert},
  \bibinfo{author}{S.~Schr{\"o}der}, \bibinfo{author}{A.~Klar},
\newblock \bibinfo{title}{A flux conserving meshfree method for conservation
  laws},
\newblock \bibinfo{journal}{International Journal for Numerical Methods in
  Engineering} \bibinfo{volume}{112} (\bibinfo{year}{2017})
  \bibinfo{pages}{238--256}.
\bibitem[{Askari and Adibi(2017)}]{askari2017numerical}
\bibinfo{author}{M.~Askari}, \bibinfo{author}{H.~Adibi},
\newblock \bibinfo{title}{Numerical solution of advection--diffusion equation
  using meshless method of lines},
\newblock \bibinfo{journal}{Iranian Journal of Science and Technology,
  Transactions A: Science} \bibinfo{volume}{41} (\bibinfo{year}{2017})
  \bibinfo{pages}{457--464}.
\bibitem[{Zhang et~al.(2016)Zhang, Zhu, Loula, and Yu}]{zhang2016operator}
\bibinfo{author}{R.~Zhang}, \bibinfo{author}{J.~Zhu}, \bibinfo{author}{A.~F.
  Loula}, \bibinfo{author}{X.~Yu},
\newblock \bibinfo{title}{Operator splitting combined with
  positivity-preserving discontinuous {G}alerkin method for the chemotaxis
  model},
\newblock \bibinfo{journal}{Journal of Computational and Applied Mathematics}
  \bibinfo{volume}{302} (\bibinfo{year}{2016}) \bibinfo{pages}{312--326}.
\bibitem[{Fiengo~P{\'e}rez et~al.(2017)Fiengo~P{\'e}rez, Sweeck, Elskens, and
  Bauwens}]{fiengo2017discontinuous}
\bibinfo{author}{F.~Fiengo~P{\'e}rez}, \bibinfo{author}{L.~Sweeck},
  \bibinfo{author}{M.~Elskens}, \bibinfo{author}{W.~Bauwens},
\newblock \bibinfo{title}{A discontinuous finite element suspended sediment
  transport model for water quality assessments in river networks},
\newblock \bibinfo{journal}{Hydrological Processes} \bibinfo{volume}{31}
  (\bibinfo{year}{2017}) \bibinfo{pages}{1804--1816}.
\bibitem[{Wang and Wu(2024)}]{wang2024discontinuous}
\bibinfo{author}{J.~Wang}, \bibinfo{author}{S.~Wu},
\newblock \bibinfo{title}{Discontinuous {G}alerkin methods for magnetic
  advection-diffusion problems},
\newblock \bibinfo{journal}{Computers \& Mathematics with Applications}
  \bibinfo{volume}{174} (\bibinfo{year}{2024}) \bibinfo{pages}{43--54}.
\bibitem[{Dal~Santo et~al.(2017)Dal~Santo, Deparis, and
  Manzoni}]{dal2017numerical}
\bibinfo{author}{N.~Dal~Santo}, \bibinfo{author}{S.~Deparis},
  \bibinfo{author}{A.~Manzoni},
\newblock \bibinfo{title}{A numerical investigation of multi space reduced
  basis preconditioners for parametrized elliptic advection-diffusion
  equations},
\newblock \bibinfo{journal}{Communications in Applied and Industrial
  Mathematics} \bibinfo{volume}{8} (\bibinfo{year}{2017})
  \bibinfo{pages}{282--297}.
\bibitem[{Karimi and Nakshatrala(2016)}]{karimi2016current}
\bibinfo{author}{S.~Karimi}, \bibinfo{author}{K.~Nakshatrala},
\newblock \bibinfo{title}{Do current lattice {B}oltzmann methods for diffusion
  and advection-diffusion equations respect maximum principle and the
  non-negative constraint?},
\newblock \bibinfo{journal}{Communications in Computational Physics}
  \bibinfo{volume}{20} (\bibinfo{year}{2016}) \bibinfo{pages}{374--404}.
\bibitem[{Yan et~al.(2017)Yan, Yang, Li, and Hilpert}]{yan2017two}
\bibinfo{author}{Z.~Yan}, \bibinfo{author}{X.~Yang}, \bibinfo{author}{S.~Li},
  \bibinfo{author}{M.~Hilpert},
\newblock \bibinfo{title}{Two-relaxation-time lattice {B}oltzmann method and
  its application to advective-diffusive-reactive transport},
\newblock \bibinfo{journal}{Advances in water resources} \bibinfo{volume}{109}
  (\bibinfo{year}{2017}) \bibinfo{pages}{333--342}.
\bibitem[{Romero et~al.(2017)Romero, Witherden, and Jameson}]{romero2017direct}
\bibinfo{author}{J.~Romero}, \bibinfo{author}{F.~D. Witherden},
  \bibinfo{author}{A.~Jameson},
\newblock \bibinfo{title}{A direct flux reconstruction scheme for
  advection--diffusion problems on triangular grids},
\newblock \bibinfo{journal}{Journal of Scientific Computing}
  \bibinfo{volume}{73} (\bibinfo{year}{2017}) \bibinfo{pages}{1115--1144}.
\bibitem[{Penenko et~al.(2017)Penenko, Penenko, and
  Mukatova}]{penenko2017direct}
\bibinfo{author}{A.~Penenko}, \bibinfo{author}{V.~Penenko},
  \bibinfo{author}{Z.~Mukatova},
\newblock \bibinfo{title}{Direct data assimilation algorithms for
  advection-diffusion models with the increased smoothness of the uncertainty
  functions},
\newblock in: \bibinfo{booktitle}{2017 International Multi-Conference on
  Engineering, Computer and Information Sciences (SIBIRCON)},
  \bibinfo{organization}{IEEE}, pp. \bibinfo{pages}{126--130}.
\bibitem[{Lewis et~al.(2017)Lewis, Coakley, and Miele}]{lewis2017extension}
\bibinfo{author}{K.~Lewis}, \bibinfo{author}{S.~Coakley},
  \bibinfo{author}{S.~Miele},
\newblock \bibinfo{title}{An extension of the {S}tefan-type solution method
  applicable to multi-component, multi-phase 1d systems},
\newblock \bibinfo{journal}{Transport in Porous Media} \bibinfo{volume}{117}
  (\bibinfo{year}{2017}) \bibinfo{pages}{415--441}.
\bibitem[{Ho et~al.(2024)Ho, Suk, Liang, Liu, Nguyen, and
  Chen}]{ho2024recursive}
\bibinfo{author}{Y.-C. Ho}, \bibinfo{author}{H.~Suk}, \bibinfo{author}{C.-P.
  Liang}, \bibinfo{author}{C.-W. Liu}, \bibinfo{author}{T.-U. Nguyen},
  \bibinfo{author}{J.-S. Chen},
\newblock \bibinfo{title}{Recursive analytical solution for nonequilibrium
  multispecies transport of decaying contaminant simultaneously coupled in both
  the dissolved and sorbed phases},
\newblock \bibinfo{journal}{Advances in Water Resources} \bibinfo{volume}{192}
  (\bibinfo{year}{2024}) \bibinfo{pages}{104777}.
\bibitem[{Kumar et~al.(2019)Kumar, Pandey, and Das}]{kumar2019gegenbauer}
\bibinfo{author}{S.~Kumar}, \bibinfo{author}{P.~Pandey},
  \bibinfo{author}{S.~Das},
\newblock \bibinfo{title}{Gegenbauer wavelet operational matrix method for
  solving variable-order non-linear reaction--diffusion and {G}alilei invariant
  advection--diffusion equations},
\newblock \bibinfo{journal}{Computational and Applied Mathematics}
  \bibinfo{volume}{38} (\bibinfo{year}{2019}) \bibinfo{pages}{162}.
\bibitem[{Nagy and Issa(2024)}]{nagy2024accurate}
\bibinfo{author}{A.~Nagy}, \bibinfo{author}{K.~Issa},
\newblock \bibinfo{title}{An accurate numerical technique for solving
  fractional advection--diffusion equation with generalized {C}aputo
  derivative},
\newblock \bibinfo{journal}{Zeitschrift f{\"u}r angewandte Mathematik und
  Physik} \bibinfo{volume}{75} (\bibinfo{year}{2024}) \bibinfo{pages}{164}.
\bibitem[{O'Sullivan(2019)}]{o2019runge}
\bibinfo{author}{S.~O'Sullivan},
\newblock \bibinfo{title}{Runge--{K}utta--{G}egenbauer explicit methods for
  advection-diffusion problems},
\newblock \bibinfo{journal}{Journal of Computational Physics}
  \bibinfo{volume}{388} (\bibinfo{year}{2019}) \bibinfo{pages}{209--223}.
\bibitem[{Elgindy(2023)}]{Elgindy2023a}
\bibinfo{author}{K.~T. Elgindy},
\newblock \bibinfo{title}{New optimal periodic control policy for the optimal
  periodic performance of a chemostat using a {F}ourier--{G}egenbauer-based
  predictor-corrector method},
\newblock \bibinfo{journal}{Journal of Process Control} \bibinfo{volume}{127}
  (\bibinfo{year}{2023}) \bibinfo{pages}{102995}.
\bibitem[{Elgindy(2024{\natexlab{a}})}]{elgindy2024numerical}
\bibinfo{author}{K.~T. Elgindy},
\newblock \bibinfo{title}{Numerical solution of nonlinear periodic optimal
  control problems using a {F}ourier integral pseudospectral method},
\newblock \bibinfo{journal}{Journal of Process Control} \bibinfo{volume}{144}
  (\bibinfo{year}{2024}{\natexlab{a}}) \bibinfo{pages}{103326}.
\bibitem[{Elgindy(2024{\natexlab{b}})}]{elgindy2024optimal}
\bibinfo{author}{K.~T. Elgindy},
\newblock \bibinfo{title}{Optimal periodic control of unmanned aerial vehicles
  based on {F}ourier integral pseudospectral and edge-detection methods},
\newblock \bibinfo{journal}{Unmanned Systems}
  (\bibinfo{year}{2024}{\natexlab{b}}).
\bibitem[{Elgindy and Smith-Miles(2013)}]{Elgindy201382}
\bibinfo{author}{K.~T. Elgindy}, \bibinfo{author}{K.~A. Smith-Miles},
\newblock \bibinfo{title}{Optimal {G}egenbauer quadrature over arbitrary
  integration nodes},
\newblock \bibinfo{journal}{Journal of Computational and Applied Mathematics}
  \bibinfo{volume}{242} (\bibinfo{year}{2013}) \bibinfo{pages}{82--106}.
\bibitem[{Elgindy and Refat(2018)}]{elgindy2018high}
\bibinfo{author}{K.~T. Elgindy}, \bibinfo{author}{H.~M. Refat},
\newblock \bibinfo{title}{High-order shifted {G}egenbauer integral
  pseudo-spectral method for solving differential equations of {L}ane--{E}mden
  type},
\newblock \bibinfo{journal}{Applied Numerical Mathematics}
  \bibinfo{volume}{128} (\bibinfo{year}{2018}) \bibinfo{pages}{98--124}.
\bibitem[{Elgindy(2018)}]{elgindy2018optimal}
\bibinfo{author}{K.~T. Elgindy},
\newblock \bibinfo{title}{Optimal control of a parabolic distributed parameter
  system using a fully exponentially convergent barycentric shifted
  {G}egenbauer integral pseudospectral method},
\newblock \bibinfo{journal}{Journal of Industrial \& Management Optimization}
  \bibinfo{volume}{14} (\bibinfo{year}{2018}) \bibinfo{pages}{473}.
\bibitem[{Elgindy(2016)}]{Elgindy20161}
\bibinfo{author}{K.~T. Elgindy},
\newblock \bibinfo{title}{High-order numerical solution of second-order
  one-dimensional hyperbolic telegraph equation using a shifted {G}egenbauer
  pseudospectral method},
\newblock \bibinfo{journal}{Numerical Methods for Partial Differential
  Equations} \bibinfo{volume}{32} (\bibinfo{year}{2016})
  \bibinfo{pages}{307--349}.
\bibitem[{Elgindy(2017)}]{Elgindy20171}
\bibinfo{author}{K.~T. Elgindy},
\newblock \bibinfo{title}{High-order, stable, and efficient pseudospectral
  method using barycentric {G}egenbauer quadratures},
\newblock \bibinfo{journal}{Applied Numerical Mathematics}
  \bibinfo{volume}{113} (\bibinfo{year}{2017}) \bibinfo{pages}{1--25}.
\bibitem[{Elgindy and Karas{\"o}zen(2019)}]{Elgindy2019b}
\bibinfo{author}{K.~T. Elgindy}, \bibinfo{author}{B.~Karas{\"o}zen},
\newblock \bibinfo{title}{High-order integral nodal discontinuous
  {G}egenbauer-{G}alerkin method for solving viscous {B}urgers' equation},
\newblock \bibinfo{journal}{International Journal of Computer Mathematics}
  \bibinfo{volume}{96} (\bibinfo{year}{2019}) \bibinfo{pages}{2039--2078}.
\bibitem[{Demir and Bildik(2012)}]{demir2012numerical}
\bibinfo{author}{D.~D. Demir}, \bibinfo{author}{N.~Bildik},
\newblock \bibinfo{title}{The numerical solution of heat problem using cubic
  {B}-splines},
\newblock \bibinfo{journal}{Applied Mathematics} \bibinfo{volume}{2}
  (\bibinfo{year}{2012}) \bibinfo{pages}{131--135}.
\bibitem[{Mohebbi and Dehghan(2010)}]{mohebbi2010high}
\bibinfo{author}{A.~Mohebbi}, \bibinfo{author}{M.~Dehghan},
\newblock \bibinfo{title}{High-order compact solution of the one-dimensional
  heat and advection--diffusion equations},
\newblock \bibinfo{journal}{Applied mathematical modelling}
  \bibinfo{volume}{34} (\bibinfo{year}{2010}) \bibinfo{pages}{3071--3084}.
\bibitem[{Sun and Zhang(2003)}]{sun2003high}
\bibinfo{author}{H.~Sun}, \bibinfo{author}{J.~Zhang},
\newblock \bibinfo{title}{A high-order compact boundary value method for
  solving one-dimensional heat equations},
\newblock \bibinfo{journal}{Numerical Methods for Partial Differential
  Equations: An International Journal} \bibinfo{volume}{19}
  (\bibinfo{year}{2003}) \bibinfo{pages}{846--857}.

\end{thebibliography}
\end{document}